\documentclass[11pt, reqno]{amsart}
\usepackage{bbm}
\usepackage[top=2.7cm,left=2.8cm,right=2.8cm,bottom=2.7cm]{geometry}
\usepackage{amsfonts}
\usepackage{amsmath,amsthm,amscd,mathrsfs,amssymb,graphicx,enumerate,
	dsfont}
\usepackage{hyperref}
\usepackage[usenames,dvipsnames]{xcolor}
\usepackage{xypic}
\usepackage{tikz-cd}

\hypersetup{colorlinks=true,citecolor=NavyBlue,linkcolor=Brown,urlcolor=Orange}

\tikzset{every picture/.style={line width=0.11mm}}
\newcommand{\oPerpSymbol}{\begin{tikzpicture}[scale=0.134]
	\draw (0,-0.5)--(0,1); \draw (-0.866,-0.5)--(0.866,-0.5);
	\draw (0,0) circle [radius=1];
	\end{tikzpicture}}
\newcommand{\oPerp}{\mathbin{\raisebox{-1pt}{\oPerpSymbol}}}

\newtheorem{thm2}{Theorem}
\newtheorem{cor2}{Corollary}

\newtheorem{thm}{Theorem}[section]

\newtheorem{lem}[thm]{Lemma}
\newtheorem{prop}[thm]{Proposition}

\theoremstyle{definition}
\newtheorem{defn}[thm]{Definition}

\newtheorem{rmk}[thm]{Remark}

\DeclareMathOperator{\ch}{char} \DeclareMathOperator{\Spec}{Spec}

\DeclareMathOperator{\Hom}{Hom}

\DeclareMathOperator{\Pic}{Pic} 

\DeclareMathOperator{\prim}{prim} 
\DeclareMathOperator{\tr}{tr} 
\DeclareMathOperator{\alg}{alg} 

\newcommand{\C}{\ensuremath\mathds{C}}

\newcommand{\Z}{\ensuremath\mathds{Z}}
\newcommand{\Q}{\ensuremath\mathds{Q}}

\newcommand{\h}{\ensuremath\mathfrak{h}}
\renewcommand{\1}{\ensuremath\mathds{1}}

\newcommand{\PP}{\ensuremath\mathds{P}}

\newcommand{\XX}{\mathcal X}
\renewcommand{\AA}{\mathcal A}
\renewcommand{\P}{\mathcal P}
\newcommand{\D}{\ensuremath\mathrm{D}}
\newcommand{\HH}{\ensuremath\mathrm{H}}

\newcommand{\CH}{\ensuremath\mathrm{CH}}

\newcommand{\LL}{\ensuremath\mathrm{L}}
\newcommand{\RR}{\ensuremath\mathrm{R}}

\newif\ifHideFoot
\HideFoottrue  
\HideFootfalse  

\ifHideFoot

\newcommand{\Lie}[1]{}
\newcommand{\Charles}[1]{}

\else 

\newcommand{\marg}[1]{\normalsize{{
			\color{red}\footnote{{\color{blue}#1}}}{\marginpar[\vskip
			-.25cm{\color{red}\hfill$\Rightarrow$\tiny\thefootnote}]{\vskip
				-.2cm{\color{red}$\Leftarrow$\tiny\thefootnote}}}}}
\newcommand{\Lie}[1]{\marg{(Lie) #1}}
\newcommand{\Charles}[1]{\marg{(Charles) #1}}

\fi

\AtBeginDocument{%
	\def\MR#1{}
}

\begin{document}

	\title[Cubic fourfolds, 
	Kuznetsov components and Chow motives]{Cubic fourfolds,\linebreak
		Kuznetsov components and Chow motives}
	
	\author{Lie Fu}
		\address{Institut de recherche math\'ematique avanc\'ee (IRMA), Universit\'e de Strasbourg, France}
	\email{lie.fu@math.unistra.fr}
	
	\author{Charles Vial}
	\address{Fakult\"at f\"ur Mathematik, Universit\"at Bielefeld, Germany} 
	\email{vial@math.uni-bielefeld.de}
	\thanks{2020 {\em Mathematics Subject Classification:} 14C25, 14C15,
		14F08, 	14J28, 14J42.}
	
	\thanks{{\em Key words and phrases.}  
		Motives, K3 surfaces, cubic fourfolds, derived categories, cohomology ring.}

	\thanks{The work of L.F.\ is supported by the University of Strasbourg Institute for Advanced Study (USIAS) and by the Agence Nationale de la Recherche (ANR), under project numbers ANR-16-CE40-0011 and ANR-20-CE40-0023. The research of Ch.V.\ is funded by the Deutsche Forschungsgemeinschaft (DFG, German Research Foundation) -- SFB-TRR 358/1 2023 -- 491392403}
	
	\begin{abstract} 
		We prove that the Chow motives of two smooth cubic fourfolds whose Kuznetsov components are Fourier--Mukai equivalent are isomorphic as Frobenius algebra objects.  
		As a corollary, there exists a Galois-equivariant isomorphism between their $\ell$-adic cohomology Frobenius algebras.
		We also discuss the case where the Kuznetsov component of a smooth cubic fourfold is equivalent to the derived category of a K3 surface.
	\end{abstract}
	
	\maketitle

	\section{Introduction}

	In \cite{FV2}, we asked whether the bounded derived category of coherent sheaves on a hyper-K\"ahler variety $X$ encodes the intersection theory on $X$ and its powers. Precisely, given two hyper-K\"ahler varieties $X$ and $X'$ that are derived-equivalent, i.e.\ $\D^b(X)\simeq \D^b(X')$, we asked whether the Chow motives with rational coefficients of $X$ and $X'$ are isomorphic as algebra objects. The main result of \cite{FV2} establishes this in the simplest case where $X$ and $X'$ are K3 surfaces. 
	The above expectation refines,  in the special case of hyper-K\"ahler varieties, a general conjecture of Orlov~\cite{MR1998775}, predicting that two derived-equivalent smooth projective varieties have isomorphic Chow motives with rational coefficients.
	
	Like hyper-K\"ahler varieties, the so-called \textit{K3-type varieties} also behave in many ways like K3 surfaces. By definition \cite{FLV2}, those are Fano varieties $X$ of even dimension $2n$ with Hodge numbers $h^{p,q}(X) = 0$ for all $p\neq q$ except for $h^{n-1,n+1}(X)= h^{n+1,n-1}(X) = 1$. 
	Some basic examples of such varieties are cubic fourfolds, Gushel--Mukai fourfolds and sixfolds \cite{Muk89, KP16}, and Debarre--Voisin 20-folds \cite{MR2746467}.
	As an important interplay between Fano varieties of K3 type and hyper-K\"ahler varieties, many hyper-K\"ahler varieties are constructed as moduli spaces of stable objects on some admissible subcategories of the derived categories of such Fano varieties \cite{BLMS, LLMS, LPZ18, LPZ19}. Due to these links, in \cite{FLV2}, we asked whether the Chow motives, considered as algebra objects, of Fano varieties of K3 type had similar properties as K3 surfaces (and  what is expected for hyper-K\"ahler varieties). 
	
	 Based on the above, we may ask whether two derived-equivalent Fano varieties of K3 type have isomorphic Chow motives as algebra objects. However, this question is uninteresting\,: due to the celebrated result of Bondal--Orlov \cite{BondalOrlov01}, any two derived-equivalent Fano varieties are isomorphic.
	In the case of a cubic fourfold $X$, Kuznetsov \cite{Kuz10} has identified an interesting admissible subcategory $\mathcal{A}_X$ of $\D^b(X)$, called the \textit{Kuznetsov component}, consisting of objects $E$ such that $\Hom(\mathcal{O}_X(i), E[m])=0$ for $i=0, 1, 2$  and any $m\in \Z$.  The Kuznetsov component is a K3-like triangulated category,
see \S \ref{subsec:KuznetsovComponent}. 
Our first main  result gives the correct analog of the aforementioned results on K3 surfaces for cubic fourfolds: two cubic fourfolds with Fourier--Mukai equivalent Kuznetsov components have isomorphic Chow motives as algebra objects. More precisely, we have the following.
	
	
	\begin{thm2}\label{T:cubics}
	Let $X$ and $X'$ be two smooth cubic fourfolds over a field $K$ with Fourier--Mukai equivalent Kuznetsov components $\AA_X \simeq \AA_{X'}$. Then $X$ and $X'$ have isomorphic Chow motives, as Frobenius algebra objects, in the category of rational Chow motives over $K$. 
	\end{thm2}

We refer to \S\ref{Subsec:MainResult} for the notion of \emph{Fourier--Mukai equivalence} for Kuznetsov components.  By \cite{LPZ-enhancement22}, if $K=\C$ and if $\AA_X $ and  $\AA_{X'}$ are equivalent as $\C$-linear triangulated categories, then they are Fourier--Mukai equivalent.

Following our previous work \cite[\S 2]{FV2}, a \textit{Frobenius algebra object} in a rigid tensor category is an algebra object together with an extra structure, namely an isomorphism to its dual object (which we call a non-degenerate quadratic space structure, see \S \ref{SS:quadratic}) with a compatibility condition. The Chow motive of any smooth projective variety carries a natural structure of Frobenius algebra object in the category of Chow motives, lifting the classical Frobenius algebra structure on the cohomology ring (which essentially consists of the cup-product $\smile$ together with the degree map $\int_X$). We refer to Section~\ref{S:frob} for more details. An immediate concrete application of Theorem~\ref{T:cubics} is the following result.

\begin{cor2}\label{Cor:main}
	Let $X$ and $X'$ be two smooth cubic fourfolds over a field $K$. Assume that their Kuznetsov components are Fourier--Mukai equivalent $\AA_X \simeq \AA_{X'}$. Then there exists a correspondence $\Gamma \in \CH^4(X\times_K X')\otimes \Q$ such that for any Weil cohomology $\HH^*$ with coefficients in a field of characteristic zero,
	$$
	\Gamma_* : \HH^*(X) \stackrel{\sim}{\longrightarrow} \HH^*(X')$$ is an isomorphism of Frobenius algebras. In particular,
	\begin{enumerate}[(i)]
\item for any prime number $\ell \neq \operatorname{char} K$, there exists a Galois-equivariant isomorphism $\HH^*(X_{\bar K},\Q_\ell)\simeq \HH^*(X'_{\bar K},\Q_\ell)$ of  $\ell$-adic cohomology Frobenius algebras\,;
\item there exists an isocrystal isomorphism $\HH^*_{\mathrm{cris}}(X)\simeq \HH^*_{\mathrm{cris}}(X')$ of crystalline cohomology Frobenius  algebras\,;
\item if $K=\C$, there exists a Hodge isomorphism $\HH^*(X, \mathbb{Q})\simeq \HH^*(X', \mathbb{Q})$ of Betti cohomology Frobenius algebras.
	\end{enumerate}
\end{cor2}

We note that item $(iii)$ can also be directly deduced from arguments due to Addington--Thomas~\cite{MR3229044} and Huybrechts~\cite{MR3705236}\,; see Remark~\ref{R:complex}.
The proof of Theorem~\ref{T:cubics} is given in \S \ref{S:mainproof} and employs essentially two different sources of techniques. On the one hand, we proceed to a refined Chow--K\"unneth decomposition (\S \ref{SS:rCK}), thereby cutting the motive of a cubic fourfold into the sum of its transcendental part and its algebraic part. The transcendental part, as well as its relation to the algebraic part, is then dealt with  via a weight argument (\S \ref{ss:Weight}), while the algebraic part is dealt with via considering the Chow ring modulo numerical equivalence (Proposition~\ref{P:cubicsmotives}).
On the other hand, our proof also relies on some cycle-theoretic properties of cubic fourfolds, in particular those recently established in \cite{FLV2, FLV3}. 
First, the so-called Franchetta property for cubic fourfolds and their squares (Proposition~\ref{T:Franchetta}) is used 
to establish the following.

\begin{thm2}[Theorem~\ref{T:GammaTrans}] \label{T:cubics-quadratic}
	Let $X$ and $X'$ be two smooth cubic fourfolds over a field $K$ with Fourier--Mukai equivalent Kuznetsov components $\AA_X \simeq \AA_{X'}$. Then the transcendental motives $\h^4_{\operatorname{tr}}(X)(2)$ and $\h^4_{\operatorname{tr}}(X')(2)$, as defined in \S \ref{SS:rCK}, are isomorphic as quadratic space objects in the category of rational Chow motives over $K$. 
\end{thm2}
 Concretely, this involves exhibiting an isomorphism $\Gamma_{\tr} : \h^4_{\tr}(X) \to \h^4_{\tr}(X')$ with inverse given by its transpose. Precisely, we show in Theorem~\ref{T:GammaTrans} that such an isomorphism is induced by the degree-4 part of the Mukai vector of the Fourier--Mukai kernel inducing the equivalence $\AA_X\simeq \AA_{X'}$. Such an isomorphism is then upgraded in Proposition~\ref{P:cubicsmotives} to an isomorphism $\Gamma : \h(X)\to \h(X')$ with inverse given by its transpose, or equivalently, to a quadratic space object isomorphism $\Gamma : \h(X)(2)\to \h(X')(2)$.
 
 The next step towards the proof of Theorem~\ref{T:cubics} consists in showing that this isomorphism $\Gamma : \h(X) \to \h(X')$ respects the algebra structure. This is achieved in Proposition~\ref{P:upgradeFrob}, the proof of which relies on the recently established \textit{multiplicative Chow--K\"unneth relation}~\eqref{E:mult} for cubic fourfolds (Theorem~\ref{thm:MCKcubic}).
 \medskip
%
%
%
%
%

To make the analogy with our previous work \cite{FV2} even more transparent,  we also investigate the case of cubic fourfolds with associated (twisted) K3 surfaces, resulting in the following strengthening of \cite[Theorem 0.4]{Buelles}.
\begin{thm2}[Theorem~\ref{T:cubicK3v3}]	\label{T:cubicK3}
Let $X$ be a smooth cubic fourfold over a field $K$ and let $S$ be a K3 surface over $K$ equipped with a Brauer class $\alpha$. Assume that $\AA_X$ and $\D^b(S,\alpha)$ are Fourier--Mukai equivalent. Then the transcendental motives $\h^4_{\tr}(X)(2)$ and $\h^2_{\tr}(S)(1)$ are isomorphic as quadratic space objects in the category of rational Chow motives over $K$. 
\end{thm2}

Note that by Orlov's result, any equivalence between $\AA_X$ and $\D^b(S,\alpha)$ is a Fourier--Mukai equivalent,  at least when $\alpha=0$.

In a similar vein to Corollary~\ref{Cor:main}, one obtains from Theorem~\ref{T:cubics-quadratic} and Theorem~\ref{T:cubicK3} respectively, after passing to any Weil cohomology theory $\HH^*$ (e.g., Betti, $\ell$-adic, crystalline), isomorphisms 
$$\HH_{\tr}^*(X) \stackrel{\sim}{\longrightarrow} \HH_{\tr}^*(X'),$$
$$\HH_{\tr}^*(X) \stackrel{\sim}{\longrightarrow} \HH_{\tr}^*(S)$$
	that are compatible with  the natural extra structures (e.g., Hodge, Galois, Frobenius) and with the quadratic form $(\alpha,\beta) \mapsto \int_X \alpha \smile \beta$.\medskip


\noindent\textbf{Conventions.} From \S \ref{SS:MCK} onwards, $\CH^*(-)$ denotes the Chow group with rational coefficients, $\overline{\CH}^*(-)$ denotes its reduction modulo numerical equivalence, and motives are with rational coefficients.\medskip

\noindent\textbf{Acknowledgments.} We thank Xiaolei Zhao for helpful discussions and we thank the referee for their thoughtful remarks.

\section{Chow motives and Frobenius algebra objects}\label{S:frob}

In this section, we fix  a commutative ring $R$.

\subsection{Chow motives} We refer to \cite[\S 4]{MR2115000} for more details.
Briefly, a \emph{Chow motive}, or motive, over a field $K$ with coefficients in $R$, is a triple $(X,p,n)$ consisting of a smooth projective variety $X$ over $K$, an idempotent correspondence $p\in \CH^{\dim X}(X\times_K X)\otimes R$, and an integer $n\in \Z$. The motive of a smooth projective variety $X$ over $K$ is the motive $\h(X) := (X,\Delta_X,0)$, where $\Delta_X$ is the class of the diagonal inside $X\times_KX$. A morphism $\Gamma : (X,p,n) \to (Y,q,m)$ between two motives is a correspondence $\Gamma \in \CH^{\dim X -n+m}(X\times_K Y) \otimes R$ such that $q\circ \Gamma \circ p = \Gamma$. The composition of morphisms is given by the composition of correspondences (as in \cite[\S 16]{MR1644323}).
The category of Chow motives $\mathcal M(K)_R$ over $K$ with coefficients in $R$ forms a $R$-linear rigid $\otimes$-category with unit $\mathds{1} = \h(\Spec K)$, with tensor product given by $(X, p, n)\otimes (Y, q, m)=(X\times_K Y, p\times q, n+m)$ and with duality given by $(X,p,n)^\vee = (X,{}^tp, \dim X -n)$, where ${}^tp$ denotes the transpose of the correspondence $p$.

 Fix a homomorphism $R\to F$ to a field $F$ and fix a Weil cohomology theory $\HH^*$ with field of coefficients $F$, i.e., a $\otimes$-functor $\HH^* : \mathcal M(K)_R \to \mathrm{GrVec}_F$ to the category of $\Z$-graded $F$-vector spaces such that $\HH^i(\mathds{1}(-1)) = 0$ for $i\neq 2$\,; see \cite[Proposition~4.2.5.1]{MR2115000}. We also call such a $\otimes$-functor an \emph{$\HH$-realization}. One thereby obtains the category of homological motives $\mathcal M_{\HH}(K)_R$ (or $\mathcal M_{\hom}(K)_R$, when $\HH$ is clear from the context).


\subsection{Algebra structure} \label{SS:algebra}
We consider the general situation where $\mathcal C$ is an $R$-linear $\otimes$-category with unit $\mathds 1$\,; cf.~\cite[\S 2.2.2]{MR2115000}. An \emph{algebra structure} on an object $M$ in $\mathcal C$ is the data consisting of a \emph{unit morphism} $\epsilon : \mathds{1} \to M$ and a \emph{multiplication morphism} $\mu : M\otimes M \to M$ satisfying the associativity axiom $\mu \circ (\mathrm{id}_M\otimes \mu) = \mu \circ (\mu\otimes \mathrm{id}_M) $ and the unit axiom $\mu\circ (\mathrm{id}_M\otimes \epsilon) = \mathrm{id}_M = \mu \circ (\epsilon \otimes \mathrm{id}_M)$. The algebra structure is said to be commutative if it satisfies the commutativity axiom $\mu\circ \tau =  \mu$ where $\tau : M\otimes M \to M\otimes M$ is the morphism permuting the two factors. \medskip

In case $\mathcal C$ is the category of Chow motives over $K$, then the Chow motive $\h(X)$ of a smooth projective variety $X$ over $K$ is naturally endowed with a commutative algebra structure\,: the multiplication $\mu : \h(X)\otimes \h(X) \to \h(X)$ is given by pulling back along the diagonal embedding $\delta_X : X \hookrightarrow X\times X$, while the unit morphism $\eta : \mathds 1 \to \h(X)$ is given by pulling back along the structure morphism $\epsilon_X : X\to \Spec K$.
Taking the $\HH$-realization, this algebra structure endows $\HH^*(X)$ with the usual super-commutative algebra structure given by cup-product.

\subsection{Quadratic space structure} \label{SS:quadratic}
We now consider the general situation where $\mathcal C$ is an $R$-linear rigid $\otimes$-category with unit $\mathds 1$ and equipped with a $\otimes$-invertible object denoted $\mathds 1(1)$. Let $d$ be an integer. A \emph{degree-$d$ quadratic space structure}, or by abuse a \emph{quadratic space structure}, on an object~$M$ of $\mathcal C$ consists of a morphism, called \emph{quadratic form}, 
 $$q : M\otimes M \to \mathds 1(-d),$$
 which is commutative $q\circ \tau=q$, where $\tau: M\otimes M\to M\otimes M$ is the switching morphism.
 We say that an object $M$ equipped with the quadratic form $q$ above is a \emph{degree-$d$ quadratic space object} in  $\mathcal C$, or by abuse a \emph{quadratic space object}.
The quadratic form $q : M\otimes M \to \mathds 1(-d)$ is said to be \emph{non-degenerate} if the induced morphism $M(d) \to M^\vee$ is an isomorphism. Here the morphism $M(d) \to M^\vee$ is obtained by tensoring $q$ with $\mathrm{id}_{M^\vee(d)}$ and pre-composing with $\mathrm{id}_{M(d)}\otimes \operatorname{coev}$, where  $\operatorname{coev}: \mathds{1} \to M\otimes M^\vee$ is the co-evaluation map.
\medskip

In case $\mathcal C$ is the category of Chow motives over $K$, then the Chow motive $\h(X)$ of a smooth projective variety $X$ of dimension $d$ over $K$ is naturally endowed with a non-degenerate degree-$d$ quadratic space structure\,: the quadratic form $q_X: \h(X) \otimes \h(X) \to \mathds{1}(-d)$ is simply given by the class of the diagonal $\Delta_X$. In relation to the natural algebra structure on $\h(X)$, we have
 $$
 \begin{tikzcd}
 q_X:	\h(X)\otimes \h(X)\arrow{r}{\mu} & \h(X) \arrow{r}{\epsilon} & \mathds{1}(-d),
 \end{tikzcd}
 $$
 where $\epsilon : \h(X) \to \mathds 1(-d)$ is the dual of the unit morphism $\eta : \mathds 1 \to \h(X)$.
Taking the $\HH$-realization, this degree-$d$ quadratic structure endows $\HH^*(X)$, as a super-vector space, with the usual quadratic structure given by 
\begin{equation}\label{E:quadratic}
\begin{tikzcd}
q_X:	\HH^*(X)\otimes \HH^*(X)\arrow{r}{\smile} & \HH^*(X) \arrow{r}{\deg} & F(-d).
\end{tikzcd}
\end{equation}
 Note that when $d$ is odd the form is anti-symmetric on $\HH^d(X)$, while when $d$ is even, the form is symmetric on $\HH^d(X)$.\medskip


%

In what follows, if $M = (X,p,d)$ is a Chow motive with $\dim X = 2d$, we view $M$ as a quadratic space object via 
$$
\begin{tikzcd}
q_M:	M\otimes M \arrow[r,hook] & \h(X)(d) \otimes \h(X)(d) \arrow{r}{\mu} & \h(X)(2d) \arrow{r}{\epsilon} & \mathds{1}.
\end{tikzcd}
$$

\begin{prop}\label{P:quadratic}
	Let $M = (X,p,d)$ and $M' = (X',p',d')$ be Chow motives in $\mathcal M(K)_R$. Assume that $p={}^tp$, $p'={}^tp'$, $\dim X  = 2d$ and $\dim X' = 2d'$, so that $M=M^\vee$ and $M'=M'^\vee$. The following are equivalent\,:
	\begin{enumerate}[(i)]
		\item $M$ and $M'$ are isomorphic as quadratic space objects\,;
		\item There exists an isomorphism $\Gamma : M \stackrel{\sim}{\longrightarrow} M'$ of Chow motives with $\Gamma^{-1} = {}^t\Gamma$.
	\end{enumerate}
\end{prop}
\begin{proof}
The quadratic forms $q_M$ and $q_{M'}$ are the (non-degenerate) quadratic forms associated to the identifications $M=M^\vee$ and $M'=M'^\vee$, respectively. By definition, a morphism $\Gamma : M\to M'$ is a morphism of quadratic space objects if and only if $q_{M'}\circ (\Gamma\otimes \Gamma) = q_M$. 
The latter is then equivalent to ${}^t\Gamma \circ \Gamma = \mathrm{id}_M$, where we have identified $\Gamma^\vee$ with ${}^t\Gamma$ via the identifications $M=M^\vee$ and $M'=M'^\vee$. This shows that a morphism $\Gamma : M\to M'$ is a morphism of quadratic space objects if and only if $\Gamma$ is split injective with left-inverse ${}^t\Gamma$. This proves the proposition.
\end{proof}

\subsection{Frobenius algebra structure} \label{subsec:Frobenius}
This notion was introduced in \cite[\S 2]{FV2}, as a generalization of the classical Frobenius algebras (cf.~\cite{MR2037238}).
Consider again the general situation where $\mathcal C$ is an $R$-linear rigid $\otimes$-category with unit $\mathds 1$ and equipped with a $\otimes$-invertible object denoted $\mathds 1(1)$. Let $d$ be an integer. A \emph{degree-$d$ (commutative) Frobenius algebra structure} on an object~$M$ of $\mathcal C$ consists of a unit morphism $\epsilon : \mathds{1} \to M$, a multiplication morphism $\mu : M\otimes M \to M$ and a non-degenerate degree-$d$ quadratic form $q : M\otimes M \to \mathds{1}(-d)$ such that $(M,\mu,\epsilon)$ is an algebra object, and the following compatibility relation, called the Frobenius condition, holds:
$$(\mathrm{id}_M\otimes \mu)\circ (\delta \otimes \mathrm{id}_M) = \delta\circ \mu = (\mu \otimes \mathrm{id}_M) \circ (\mathrm{id}_M\otimes \delta),$$
where $\delta : M \to M\otimes M(d)$ is the dual of the multiplication $\mu$, via the identification $M(d)\simeq M^\vee$ provided by the non-degenerate quadratic form $q$.
\medskip

In case $\mathcal C$ is the category of Chow motives over $K$, then the Chow motive $\h(X)$ of a smooth projective variety $X$ of dimension $d$ over $K$ is naturally endowed with a degree-$d$ Frobenius algebra structure. That the unit, multiplication and quadratic form given in \S \S \ref{SS:algebra}-\ref{SS:quadratic} above do define such a structure on $\h(X)$ is explained in \cite[Lemma~2.7]{FV2}.
Taking the $\HH$-realization and forgetting Tate twists, this degree-$d$ Frobenius algebra structure endows $\HH^*(X)$ with the usual Frobenius algebra structure (consisting of the cup-product together with the quadratic form $q_X$ of~\eqref{E:quadratic})\,; see  \cite[Example~2.5]{FV2}.
%
%

\section{The Chow ring of powers of cubic fourfolds}\label{SS:MCK}

In this section, we gather the cycle-theoretic results needed about cubic fourfolds\,; Proposition~\ref{T:Franchetta} is used to obtain isomorphisms as quadratic space objects as in Theorem~\ref{T:cubics-quadratic}, and Theorem~\ref{thm:MCKcubic} is used in addition to upgrade those isomorphisms to isomorphisms of algebra objects as in Theorem~\ref{T:cubics}.\medskip

From now on, we fix a field $K$ with algebraic closure $\bar K$, Chow groups and motives are with rational coefficients ($R=\Q$), and we fix a Weil cohomology theory $\HH^*$ with coefficients in a field of characteristic zero.
\medskip

Recall that a \emph{Chow--K\"unneth decomposition}, or \emph{weight decomposition}, for a motive $M$ is a finite grading $M = \bigoplus_{i\in \Z} M^i$ such that $\HH^*(M^i) = \HH^i(M)$. This notion was introduced by Murre~\cite{Murre}, who conjectured that every motive admits such a decomposition.
Now, if $M$ is a Chow motive equipped with an algebra structure (e.g., $M=\h(X)$ equipped with the intersection pairing), then we say that a Chow--K\"unneth decomposition $M=\bigoplus_{i\in \Z} M^i$ is \emph{multiplicative} if it defines an algebra grading, i.e., if the composition $M^i\otimes M^j \hookrightarrow M\otimes M \to M$ factors through $M^{i+j}$ for all $i,j$. This notion was introduced in \cite[\S 8]{SV}, where it was conjectured that the motive of any hyper-K\"ahler variety admits a multiplicative Chow--K\"unneth decomposition.\medskip

	Let $B$ be the open subset of $\PP\HH^0(\PP^5,\mathcal{O}(3))$
	parameterizing smooth cubic fourfolds, let $\XX \to B$ be the universal family of smooth cubic fourfolds and $\operatorname{ev}:\XX\to \PP^{5}$ be the evaluation map. If $H:=\operatorname{ev}^{*}(c_{1}(\mathcal{O}_{\PP^{5}}(1))) \in \CH^1(\XX)$ denotes the relative hyperplane section, then 
	\begin{equation}\label{eq:CKcubic}
	\pi^0_\XX = \frac{1}{3} H^4 \times_B \XX, 	\quad	\pi^2_\XX = \frac{1}{3} H^3 \times_B H, \quad		\pi^6_\XX = \frac{1}{3} H \times_B H^3, 	\quad	\pi^8_\XX = \frac{1}{3} \XX \times_B H^4\\
	\end{equation}
	\begin{equation*}
	\text{and} \quad \pi^4_\XX = \Delta_{\XX/B} - 	\pi^0_\XX - 	\pi^2_\XX - 	\pi^6_\XX -	\pi^8_\XX
	\end{equation*}
	defines a relative Chow--K\"unneth decomposition, in the sense that its specialization to any fiber $\XX_b$ over $b \in B$ gives a Chow--K\"unneth decomposition
	of $\XX_b$. Given a smooth cubic fourfold $X$, we denote $h_X$ the restriction of $H$ to $X$ and we denote $\{\pi_X^0, \pi_X^2, \pi_X^4,\pi_X^6, \pi_X^8\}$ the restriction of the above projectors to the fiber $X$.
	\medskip

	In our previous work \cite{FLV2}, we established the following two results\,:
	
	\begin{thm}\label{thm:MCKcubic}
		The  Chow--K\"unneth decomposition $\{\pi_X^0, \pi_X^2, \pi_X^4,\pi_X^6, \pi_X^8\}$  is \emph{multiplicative}. Equivalently,  in $\CH^8(X\times X\times X)$, we have
		\begin{align}\label{E:mult}
		\delta_X = & \  \frac{1}{3} \big( p_{12}^*\Delta_X\cdot p_3^*h_X^4 + p_{13}^*\Delta_X\cdot p_2^*h_X^4 + p_{23}^*\Delta_X \cdot p_1^*h_X^4 \big)+  P\big(p_{1}^*h_X, p_2^*h_X, p_3^*h _X\big),
		\end{align} 
		where $P$ is an explicit symmetric rational polynomial in 3 variables.
	\end{thm}
\begin{proof} That the Chow--K\"unneth decomposition $\{\pi_X^0, \pi_X^2, \pi_X^4,\pi_X^6, \pi_X^8\}$  is multiplicative is \cite[Corollary~1]{FLV2}. The identity \eqref{E:mult} is due to Diaz~\cite{Diaz}. That the two formulations are equivalent is \cite[Proposition~2.8]{FLV3}. The proof in \textit{loc.~cit.~}is over $\mathbb{C}$, but  one can extend the result to arbitrary base fields as follows. By the Lefschetz principle, \eqref{E:mult} holds for any algebraically closed field of characteristic zero. Since the pull-back morphism $\CH(X^3)\to \CH(X^3_\Omega)$  associated with the field extension from $K$ to a universal domain $\Omega$ is injective, and all the terms in \eqref{E:mult} are defined over $K$, we have the result in characteristic zero. If $\ch(K)>0$, take a lifting $\mathcal{X}/W$ over some discrete valuation ring $W$ with residue field $K$ and fraction field of characteristic zero. Then by specialization, the validity of \eqref{E:mult} on the generic fiber implies the same result on the special fiber. 
\end{proof}
	
%
	
%
%
%
	
	\begin{prop}\label{T:Franchetta}
Let $\XX\to B$ be the above-defined family of smooth cubic fourfolds and let $X=X_b$ be a fiber. For a positive integer $n$, define $\operatorname{GDCH}_B^*(X^n)$, which stands for generically defined cycles, to be the image of the Gysin restriction ring homomorphism
$$\CH^*(\XX_{/B}^n) \to \CH^*(X^n).$$ Then the map $\operatorname{GDCH}_B^*(X^n) \hookrightarrow \CH^*(X^n) \twoheadrightarrow \overline{\CH}\,^*(X^n)$ is injective for $n\leq 2$.
We say that $\XX_{/B}^n$ has the \emph{Franchetta property} for $n\leq 2$.
	\end{prop}
\begin{proof}
This was established in \cite[Proposition~5.6]{FLV2}. The proof in \emph{loc.\ cit.}\ is given for $K=\C$ but holds for any field $K$.
\end{proof}

\begin{rmk}
Proposition~\ref{T:Franchetta} was extended to $n\leq 4$ in \cite[Theorem~2]{FLV3}. However, the cases $n=3$ and $n=4$ are not needed for the proof of Theorem~\ref{T:cubics} and, besides, their proofs are significantly more involved.
\end{rmk}

	\section{Kuznetsov components and primitive motives}\label{sec:Projectors}

	\subsection{Kuznetsov component and projectors}
	\label{subsec:KuznetsovComponent}
	For the basic theory of Fourier--Mukai transforms, we refer to the book \cite{MR2244106}.
	Let $X\subset \PP^{5}$ be a smooth cubic fourfold defined over a base field $K$. Following \cite{Kuz10}, the \emph{Kuznetsov component} $\AA_X$ of $X$ is defined to be the right-orthogonal complement of the triangulated subcategory generated by the exceptional collection $ \langle \mathcal O_X, \mathcal O_X(1), \mathcal O_X(2)\rangle$ in the bounded derived category of coherent sheaves $\D^b(X)$:
	$$\AA_{X}:=\{E\in \D^{b}(X)~\mid~ \Hom(\mathcal{O}_{X}(i), E[k])=0, \text{ for all } i=0, 1, 2 \text{ and } k\in \mathbb{Z}\}.$$ By Serre duality, $\AA_{X}$ is also the left-orthogonal complement
	 of the triangulated subcategory generated by the exceptional collection $ \langle \mathcal O_X(-3), \mathcal O_X(-2), \mathcal O_X(-1)\rangle$ in $\D^{b}(X)$:
	$$\AA_{X}=\{E\in \D^{b}(X)~\mid~ \Hom(E[k], \mathcal{O}_{X}(i))=0, \text{ for all } i=-1, -2, -3 \text{ and } k\in \mathbb{Z}\}.$$
	In other words, we have semi-orthogonal decompositions
	$$ \D^b(X) = \langle \AA_X, \mathcal O_X, \mathcal O_X(1), \mathcal O_X(2)\rangle \quad \text{and} \quad
	 \D^b(X) = \langle \mathcal O_X(-3), \mathcal O_X(-2), \mathcal O_X(-1), \AA_X\rangle.$$
	 
	 As is pointed out by Kuznetsov \cite{Kuz10} (see also \cite[Proposition 1.4]{Kuznetsov-LectureNotes}), $\AA_X$ is a K3-like category (or sometimes called a \textit{non-commutative K3 surface}), 
	 in the sense that its Serre functor $\mathbf{S}_{\AA_X}=[2]$ (see for example \cite{Kuznetsov-FractionalCY}) and  its Hochschild homology, which is $\operatorname{HH}_*(\AA_X)=K[2]\oplus K^{22}\oplus K[-2]$, agrees with the Hochschild homology of a K3 surface, at least when $\operatorname{char}(K)\neq 2$ or $3$. The latter, which will not be used in this work, can be established by using the additivity of Hochschild homology, the HKR isomorphism \cite{Yikutieli, AntieauVezzosi-HKR} applied to the cubic fourfold, and the computation of Hodge numbers of cubic fourfolds.

As $\AA_{X}$ is an admissible subcategory (\cite{MR992977, BondalKapranov}), the inclusion functor $i_X : \AA_X \hookrightarrow \D^b(X)$ has both left and right adjoint functors\,; these are  denoted by $i_X^*$ and $i_X^! \colon \D^b(X) \to \AA_X$, respectively.  
In addition, since $i_X$ is fully faithful, the adjunction morphisms 
$$
\begin{tikzcd}
i_X^* \circ i_X \arrow{r}{\simeq} & \mathrm{id}_{\AA_X}\arrow{r}{\simeq} & i_X^!\circ i_X
\end{tikzcd}
$$
are isomorphisms. 
We then have the following basic property.
\begin{prop}\label{P:proj}
The functors $p_X^L := i_X \circ i_X^*$ and $p_X^R:= i_X\circ i_X^!$ are idempotent endo-functors of $\D^b(X)$, that is, 
$$\left \lbrace \begin{aligned}
p_X^L \circ p_X^L \simeq p_X^L\,; \\
p_X^R \circ p_X^R \simeq p_X^R.
\end{aligned}\right.$$
Moreover, we have 
$$\left \lbrace \begin{aligned}
p_X^L \circ p_X^R \simeq p_X^R\,; \\
p_X^R \circ p_X^L \simeq p_X^L.
\end{aligned}\right.$$
\end{prop}

Note that $p_{X}^{L}$ and $p_{X}^{R}$ are mutation functors in the sense of Bondal \cite{MR992977}. More precisely,
$$p_{X}^{L}=\LL_{\langle\mathcal O_X, \mathcal O_X(1), \mathcal O_X(2)\rangle} = \LL_{\mathcal O_X}\circ\LL_{\mathcal O_X(1)}\circ\LL_{\mathcal O_X(2)}$$ is the left mutation through $\langle \mathcal O_X, \mathcal O_X(1), \mathcal O_X(2)\rangle $ and
 $$p_{X}^{R}=\RR_{\langle\mathcal O_X(-3), \mathcal O_X(-2), \mathcal O_X(-1)\rangle} = \RR_{\mathcal O_X(-1)}\circ\RR_{\mathcal O_X(-2)}\circ\RR_{\mathcal O_X(-3)}$$ is the right mutation through $\langle\mathcal O_X(-3), \mathcal O_X(-2), \mathcal O_X(-1)\rangle$.

We denote $\P_X^L$ and $\P_X^R$ the respective Fourier--Mukai kernels in $\D^b(X\times_K X)$ of the functors $p_X^L$ and $p_X^L$. 
Recall that, given $E \in \D^b(X)$ an exceptional object, the Fourier--Mukai kernel of the left mutation functor $\LL_{E}$ is given by $\mathrm{cone}\big(E^\vee \boxtimes E \to  \mathcal O_{\Delta}\big)$, while the Fourier--Mukai kernel of the right mutation functor $\RR_{E}$ is given by
$\mathrm{cone}\big(\mathcal O_{\Delta} \to \mathrm{R}{\mathscr Hom} (E,\omega_X[d]) \boxtimes E \big)[-1]$.
Here, $d$ is the dimension of $X$ and $E_{1}\boxtimes E_{2}:= p^{*}E_{1}\otimes q^{*}E_{2}$ with $p,q : X\times_K X \to X$ the two natural projections.
Therefore the Fourier--Mukai kernel of $p_{X}^{L}$ is given by the convolution of the kernels of the mutation functors\,:
\begin{equation}\label{eqn:FMKernelPL}
	\P_{X}^{L}\simeq \mathrm{cone}\big(\mathcal{O}_{X\times_K X} \to  \mathcal O_{\Delta}\big) \ast \mathrm{cone}\big(\mathcal{O}_{X}(-1) \boxtimes \mathcal{O}_{X}(1) \to  \mathcal O_{\Delta}\big) \ast \mathrm{cone}\big(\mathcal{O}_{X}(-2) \boxtimes \mathcal{O}_{X}(2) \to  \mathcal O_{\Delta}\big).
\end{equation}
The Fourier--Mukai kernel $\P_{X}^{R}$ of $p_{X}^{R}$ admits a similar description. 

\begin{rmk}\label{rmk:PLfamilywise}
	Consider the universal family of smooth cubic fourfolds  $\XX \to B$ as in Section~\ref{SS:MCK}.
	Since objects of the form $\mathcal O_X(i)$ are defined family-wise for $\XX \to B$, by the formula \eqref{eqn:FMKernelPL}, the Fourier--Mukai kernels $\P_X^L$ (and similarly $\P_X^R$) are defined family-wise.
\end{rmk} 

Now we turn to the study of cohomological  or Chow-theoretic Fourier--Mukai transforms. Recall that for $E\in \D^b(X)$, its \textit{Mukai vector} is defined as $v(E):= \operatorname{ch}(E)\sqrt{\operatorname{td}(T_X)} \in \CH^*(X)$, and we denote its cohomology class by $[v(E)]\in \HH^*(X)$ and its numerical class by $\bar{v}(E)\in \overline{\CH}^*(X)$, where $\overline{\CH}^*(X):=\CH^*(X)/\!\!\equiv$ is the $\mathbb{Q}$-algebra of cycles on $X$ modulo numerical equivalence. The \textit{Mukai pairing} on $\CH^*(X)$ is given as follows: for any $v, v'\in \CH^*(X)$, 
\begin{equation}\label{E:Mukai}
\langle v, v'\rangle:=\int_{X} v^{\vee}\cdot v'\cdot\exp(c_{1}(X)/2),
\end{equation}
where $v^\vee:=\sum_{i=0}^{\dim X} (-1)^i v_i$, where $v_i\in \CH^i(X)$ is the codimension $i$ component of $v$. The same formula defines the Mukai pairing on $\HH^*(X)$ and $\overline{\CH}^*(X)$. Note that the Mukai pairing is bilinear but in general \textit{not} symmetric, hence we need to distinguish between the notions of left and right orthogonal complements. Recall that for a vector space $V$ equipped with a bilinear form $\langle -, - \rangle$, the left (resp.~right) orthogonal complement of a subspace $U$ is by definition ${}^\perp U:=\{v\in V~|~ \langle v, u \rangle=0, \text{ for all  } u\in U\}$, resp.~$ U^{\perp}:=\{v\in V~|~ \langle u, v \rangle=0, \text{ for all  } u\in U\}$. When the bilinear form is non-degenerate, we define the  \textit{orthogonal projection} from $V$ onto ${}^\perp U$ (resp.~$U^\perp$) as the projection with respect to the decomposition $V=U\oplus {}^\perp U$ (resp.~$V=U^\perp\oplus U$).
\begin{lem}\label{lemma:ProjectorPL}
	The cohomological (resp.~numerical) Fourier--Mukai transform 
	\[[v(\P_X^L)]_*: \HH^*(X)\to \HH^*(X),\]
	\[\bar{v}(\P_X^L)_*: \overline{\CH}^*(X)\to \overline{\CH}^*(X)\]
	are respectively the orthogonal projections onto $\langle v(\mathcal{O}),  v(\mathcal{O}(1)),  v(\mathcal{O}(2))\rangle^\perp$, which is the right orthogonal complement of the linear subspace spanned by the cohomological (resp.~numerical) Mukai vectors of $\mathcal{O}_{X}, \mathcal{O}_{X}(1)$, and $\mathcal{O}_{X}(2)$, with respect to the Mukai pairing.
\end{lem}
\begin{proof}
	We only show the statement for the cohomology. The proof for $\overline{\CH}^*$ is the same.
We first show a general result\,: for a smooth projective variety $X$ and an exceptional object $E$ in $\D^b(X)$, the cohomological action of the left mutation functor $\LL_E$ on $\HH^*(X)$ is the orthogonal projection onto the subspace $[v(E)]^\perp$, with respect to the Mukai pairing. Indeed, the Fourier--Mukai kernel of $\LL_E$, denoted by $F\in \D^b(X\times X)$, fits into the distinguished triangle:
\[E^\vee\boxtimes E\to \mathcal{O}_\Delta\to F\xrightarrow{+1}.\]
Hence $v(F)=v(\mathcal{O}_\Delta)-v(E^\vee\boxtimes E)=\Delta_X-v(E^\vee)\times v(E)$. Thus, for any $\alpha\in \HH^*(X)$, 
\[[v(F)]_*(\alpha)=\Delta_{X,*}(\alpha)-\left(\int_X[v(E^\vee)]\smile \alpha\right)[v(E)]=\alpha-\langle [v(E)], \alpha\rangle [v(E)],\]
which is exactly the orthogonal projector to $[v(E)]^\perp$,
where we used in the last step the relation $v(E^\vee)=v(E)^\vee\smile \exp(c_1(X)/2)$\,; see \cite[Lemma 5.41]{MR2244106}. Now back to the case of cubic fourfolds\,: since $\P_X^L$ is the composition of the kernels of three left mutations \eqref{eqn:FMKernelPL}, applying the above general result three times, we see that the cohomological transform $[v(\P_X^L)]_*$ on $\HH^*(X)$ is the successive orthogonal projections onto $[v(\mathcal{O}_X(2))]^\perp$, $[v(\mathcal{O}_X(1))]^\perp$ and $[v(\mathcal{O}_X)]^\perp$. Since  $\langle[v(\mathcal{O}_X(i))], [v(\mathcal{O}_X(j))] \rangle=0$ for all $0\leq j<i\leq 2$,  the composition of the three projections is the orthogonal projection onto  $\langle [v(\mathcal{O}_X)],  [v(\mathcal{O}_X(1))],  [v(\mathcal{O}_X(2))]\rangle^\perp$.
\end{proof}

\begin{defn} The cohomology and the Chow group modulo numerical equivalence of the Kuznetsov component $\mathcal A_X$ are defined, respectively, as the vector spaces
	\[\HH(\mathcal A_X) := \operatorname{Im} \Big([v(\P_X^L)]_*: \HH^*(X)\to \HH^*(X)\Big),\]
\[\overline{\CH}(\mathcal A_X) := \operatorname{Im}\Big(\bar{v}(\P_X^L)_*: \overline{\CH}^*(X)\to \overline{\CH}^*(X)\Big)  =\{\bar{v}(E)~|~E\in \mathcal{A}_X\}.\]
Unlike the Mukai pairing on $\HH^*(X)$ or $\overline{\CH}^*(X)$,  
the restriction of the Mukai pairing to the above spaces becomes symmetric. This holds essentially because the Serre functor 
$\mathbf{S}_{\mathcal{A}_X}$ of $\mathcal{A}_X$ is the double shift: $\langle\bar{v}(E), \bar{v}(E')\rangle=\chi(E, E')=\chi(E', \mathbf{S}_{\mathcal{A}}E)=\chi(E', E)=\langle\bar{v}(E'), \bar{v}(E)\rangle$, see \cite[pp.1891-1892]{MR3229044}. This can also be checked directly by applying the Mukai pairing to the projections of two vectors. Thus the Mukai pairing endows both $\HH(\AA_X)$ and $\overline{\CH}(\AA_X)$ with a non-degenerate quadratic form.
\end{defn}

	\subsection{Kuznetsov components and primitive classes} 
	
	\begin{defn}
Let $X$ be a smooth cubic fourfold with hyperplane class $h_X$.  The primitive cohomology and the primitive Chow group modulo numerical equivalence of $X$ are defined, respectively, to be  $$\HH^4_{\prim} (X) := \langle h_X^2 \rangle ^\perp \subseteq \HH^4(X),$$
$$\overline{\CH}\,^2_{\prim}(X) := \langle h_X^2 \rangle ^\perp \subseteq \overline{\CH}\,^2(X).$$
Here, $\langle h_X^2 \rangle ^\perp$ denotes the orthogonal complement of $h_X^2$ inside $\HH^4(X)$ with respect to the intersection product. 	We also have the following alternative description for the space of primitive classes as the right orthogonal complement of all powers of the hyperplane class\,:
$$\HH^4_{\prim}(X) = \langle \1_X, h_X, h_X^2, h_X^3, h_X^4 \rangle^\perp\subset \overline \HH^*(X),$$
$$\overline{\CH}\,^2_{\prim}(X) = \langle \1_X, h_X, h_X^2, h_X^3, h_X^4 \rangle^\perp\subset \overline{\CH}\,^*(X).$$
The restriction of the Mukai pairing on $\HH^4_{\prim} (X)$ and on $\overline{\CH}\,^2_{\prim}(X)$ endows those spaces with a non-degenerate quadratic form that coincides with the intersection pairing. (As can readily be observed from~\eqref{E:Mukai}, the Mukai pairing and the intersection pairing already agree on $\HH^4(X)$ and on $\overline{\CH}\,^2(X)$.) 
	\end{defn}

	\begin{prop}\label{P:HprimKuz}
We have the inclusions\,:
	 \begin{equation}\label{eqn:HprimKuz}
\HH^{4}_{\prim}(X)\subset \HH(\mathcal A_X),
\end{equation}
 \begin{equation}\label{eqn:HprimKuz2}
\overline{\CH}\,^2_{\prim}(X) \subset \overline{\CH}(\mathcal A_X).
\end{equation}
	\end{prop}
\begin{proof}
	We only prove~\eqref{eqn:HprimKuz} as the proof of~\eqref{eqn:HprimKuz2} is similar.
By Lemma \ref{lemma:ProjectorPL}, the right-hand side of~\eqref{eqn:HprimKuz} coincides with the right orthogonal complement of the Mukai vectors of $\mathcal{O}_{X}, \mathcal{O}_{X}(1)$, and $\mathcal{O}_{X}(2)$, with respect to the Mukai pairing on $\HH^{*}(X)$. Therefore, it suffices to check that $\HH^4_{\prim}(X)$ is right orthogonal to $[v(\mathcal{O}_{X})]$, $[v(\mathcal{O}_{X}(1))]$ and $[v(\mathcal{O}_{X}(2))]$. As the Mukai vector of the sheaf $\mathcal{O}_{X}(i)$ and $\exp(c_{1}(X)/2)$ are all polynomials in the hyperplane section class $h_X$, we have that for any $i$ there is some rational number $\lambda_{i}$ such that $$\langle  [v(\mathcal{O}_{X}(i))], \alpha\rangle=\int_{X}\alpha\smile \lambda_{i}h_X^{2}=0, \quad \forall \alpha\in \HH^{4}_{\prim}(X).$$ 
 The inclusion \eqref{eqn:HprimKuz} is proved. 
\end{proof}

\begin{rmk} 
	Over the complex numbers ($K=\mathbb{C}$), following Addington--Thomas \cite{MR3229044}, define the \emph{Mukai lattice} of $\AA_{X}$ as its topological K-theory\,:
	$$\widetilde\HH(\AA_{X}, \Z):=K_{\operatorname{top}}(\AA_{X}):=\{\alpha\in K_{\operatorname{top}}(X)~\mid~ \langle[\mathcal{O}_{X}(i)], \alpha\rangle=0 \text{ for } i=0, 1, 2\},$$
	where $\langle -, - \rangle$ is the Mukai pairing on $K_{\operatorname{top}}(X)$ given by $\langle v, v'\rangle:=\chi(v, v')$. A weight-2 Hodge structure on $\widetilde\HH(\AA_{X}, \Z)$ is induced from the isomorphism $v: K_{\operatorname{top}}(X)\otimes \Q\to \HH^{*}(X, \Q)$ given by the Mukai vector. The cohomological action of the projector $\P_{X}^{L}$ recovers the Mukai lattice rationally\,: 
	\begin{equation}\label{E:Htilde}
	\widetilde\HH(\AA_{X}, \Q)=\operatorname{Im}\left([v(\P_{X}^{L})]_{*}\colon \HH^{*}(X, \Q)\to \HH^{*}(X, \Q)\right).
	\end{equation}
	Hence Proposition \ref{P:HprimKuz} says that $\HH^{4}_{\prim}(X,\Q)\subset \widetilde\HH(\AA_{X}, \Q)$. See  \cite[Proposition 2.3]{MR3229044} for an alternative argument. 
\end{rmk}

The following relation between $\overline\CH^2_{\prim}(X)$ and $\overline{\CH}(\mathcal{A}_X)$ is essentially due to Addington--Thomas   \cite[Proposition 2.3]{MR3229044}.
\begin{prop}\label{P:num}
	There are canonical polynomials $\lambda_1$, $\lambda_2 \in \mathbb{Q}[T]$ such that we have orthogonal decompositions
	 \begin{equation}\label{eqn:HprimKuz3}
	\langle \lambda_1([h_X]),\lambda_2([h_X])\rangle \oPerp \HH^{4}_{\prim}(X) =  \HH(\mathcal A_X),
	\end{equation}
	\begin{equation}\label{eqn:HprimKuz4}
		\langle \lambda_1(h_X),\lambda_2(h_X)\rangle \oPerp \overline{\CH}\,^2_{\prim}(X) = \overline{\CH}(\mathcal A_X).
	\end{equation}
	with respect to (the restriction of) the Mukai pairing~\eqref{E:Mukai}. Moreover, the $\Z$-lattice 	$\langle \lambda_1(h_X),\lambda_2(h_X)\rangle$ equipped with the Mukai pairing is an $A_2$-lattice.
\end{prop}
\begin{proof} The decomposition
	\eqref{eqn:HprimKuz3} is established in \cite[Proposition 2.3]{MR3229044}. 
	We sketch the proof of~\eqref{eqn:HprimKuz4} for the convenience of the reader. 
	We define the polynomials (see \cite[pp.~176-177]{Huy2019Lecture})
	\begin{align*}
	\lambda_1&=3 +\frac{5}{4}T-\frac{7}{32}T^2-\frac{77}{384}T^3+\frac{41}{2048}T^4\,;\\
	\lambda_2&=-3-\frac{1}{4}T+\frac{15}{32}T^2+\frac{1}{384}T^3-\frac{153}{2048} T^4.
	\end{align*}
	We write $\lambda_i$ for $\lambda_i(h_X)$ in the sequel\,; $\lambda_i$ clearly defines an algebraic cycle defined over $K$. Let us mention that, geometrically (after a finite base-change),
	$\lambda_i$ agrees with the Mukai vector of $p^L_X(\mathcal{O}_{l}(i))$, where $l$ is any line contained in $X$.
	It is easy to compute that $\lambda_1^2=\lambda_2^2=-2$ and $\langle \lambda_1, \lambda_2\rangle=1$. Now for any element in $\overline{\CH}(\mathcal{A}_X)$, which is necessarily of the form $\bar{v}(E)$ for some $E\in \mathcal{A}_X$, the condition that $\langle\lambda_1, \lambda_2\rangle \perp \bar{v}(E)$ is equivalent to $\bar{v}(E)$ being right orthogonal in $\overline{\CH}^*(X)$ to
	\begin{align*}
	&\langle\bar{v}(\mathcal{O}_X), \bar{v}(\mathcal{O}_X(1)), \bar{v}(\mathcal{O}_X(2)), \lambda_1, \lambda_2\rangle\\
	=&\langle\bar{v}(\mathcal{O}_X), \bar{v}(\mathcal{O}_X(1)), \bar{v}(\mathcal{O}_X(2)), \bar{v}(\mathcal{O}_X(3)), \bar{v}(\mathcal{O}_X(4))\rangle\\
	=&\langle \1_X, h_X, h_X^2, h_X^3, h_X^4 \rangle.
	\end{align*} 
	However, $\langle \1_X, h_X, h_X^2, h_X^3, h_X^4 \rangle^\perp=\CH^2_{\prim}(X)$.	
\end{proof}
	
	\subsection{Kuznetsov components and primitive motives}\label{SS:Kprim}
	Let $\XX \to B$ be the universal family of smooth cubic fourfolds. We may refine the relative Chow--K\"unneth decomposition \eqref{eq:CKcubic} and define the relative idempotent correspondence
$$\pi^4_{\XX, \prim} : = \pi^4_{\XX} - \frac{1}{3} H^2\times_B H^2.$$ We have $\pi^4_{\XX, \prim} \circ \pi^4_{\XX} = \pi^4_{\XX} \circ \pi^4_{\XX, \prim} = \pi^4_{\XX, \prim}$ and 
the restriction of $\pi^4_{\XX, \prim}$ to any fiber~$X$ defines an idempotent $\pi^4_{\prim} \in \CH^4(X\times_K X)$ which cohomologically defines the orthogonal projector on the primitive cohomology $\HH^4_{\prim} (X) $.
\medskip

	Using the Franchetta property for $X\times X$ of Proposition~\ref{T:Franchetta}, we can show that the Fourier--Mukai kernels $\P_X^L$ and $\P_X^R$ enjoy the following property relatively to the projector~$\pi^4_{\prim} $. For an object $\mathcal{F}\in \D^{b}(X\times X)$, we denote by $v(\mathcal{F}):=\operatorname{ch}(\mathcal{F})\cdot \sqrt{\operatorname{td}(X\times X)}$ its Mukai vector and $v_{i}(\mathcal{F})$ the component of $v(\mathcal{F})$ in $\CH^{i}(X\times X)$, for all $0\leq i\leq 8$.
	
	\begin{lem} \label{L:key}
		The following relations hold in $\CH^4(X\times X)$\,:
		$$\pi^4_{\prim} \circ v_4(\P_X^L) \circ \pi^4_{\prim} = \pi^4_{\prim} \quad  \text{and} \quad \pi^4_{\prim} \circ v_4(\P_X^R) \circ \pi^4_{\prim} = \pi^4_{\prim}.$$
	\end{lem}
\begin{proof} We only prove the relation involving $\P_X^L$\,; the proof of the relation involving $\P_X^R$ is similar. 
We have to show	that the composition 
	\begin{equation}\label{eqn:compositionId}
	 \begin{tikzcd}
\h^4_{\prim}(X)  \arrow[hook]{r} & \h(X) \arrow{r}{v(\P_X^L)} & \bigoplus_i \h(X)(i) \arrow[twoheadrightarrow]{r} &\h^4_{\prim}(X)
	\end{tikzcd}
	\end{equation}
	 is the identity map.
	 Observe that $\pi^4_{\prim}$ is defined family-wise (which is the reason for focusing on $\pi^4_{\prim}$, rather than on $\pi^4_{\tr}$, in this section) and the Fourier--Mukai kernel $\P_X^L$ is also defined family-wise (Remark \ref{rmk:PLfamilywise}), by the Franchetta property for $X\times_K X$ in Proposition~\ref{T:Franchetta}, we are reduced to showing that the composition \eqref{eqn:compositionId} is the identity map modulo homological (or numerical) equivalence.
This follows directly from Proposition~\ref{P:HprimKuz}.
	\end{proof}

\begin{rmk}
	It is maybe possible to prove Lemma~\ref{L:key} by a direct but tedious computation without using the Franchetta property. We leave the details to the interested reader.
\end{rmk}

	\section{Equivalent Kuznetsov components and transcendental motives}
	\label{S:FM-tr}
	
	\subsection{Rational and numerical equivalence on codimension-2 cycles on cubic fourfolds} 
Recall that a \emph{universal domain} is an algebraically closed field of infinite transcendence degree over its prime subfield.
The following lemma applies in particular to cubic fourfolds\,:
	\begin{lem}\label{L:ratnum}
		Let $X$ be a smooth projective variety over a field $K$ and let $\Omega$ be a universal domain containing $K$. Assume that $\CH_0(X_\Omega)$ is supported on a curve and that  $\HH^3(X_{\bar K},\Q_\ell)=0$ for some prime $\ell \neq \operatorname{char} K$. Then rational and numerical equivalence agree on $\mathrm{Z}^2(X)$, where $\mathrm{Z}^2$ denotes the group of algebraic cycles of codimension 2 with rational coefficients. 
	\end{lem}
\begin{proof} By a push-pull argument, we may assume that $K$ is algebraically closed. 
	The proof is classical and goes back to \cite{BS}. By \cite[Proposition~1]{BS}, there exists a positive integer $N$, a $1$-dimensional closed subscheme $C\subseteq X$, a divisor $D\subset X$ and cycles $\Gamma_1, \Gamma_2$ in $\CH_{\Z}^{\dim X}(X\times_K X)$ with respective supports contained in $C\times X$ and $X\times D$, such that $$N\Delta_X = \Gamma_1 + \Gamma_2 \quad \in \CH^{\dim X}_{\Z}(X\times_K X),$$
where $\CH^*_{\Z}$ denotes the Chow group with integral coefficients.	Let $\widetilde{D} \to D$ be an alteration, say of degree $d$, with $\widetilde D$ smooth over $K$.
The multiplication by $Nd$ map on $\CH_{\Z}^2(X)$ then factors as 
\begin{equation}\label{E:diag-cor}
\begin{tikzcd}
\CH^2_{\Z}(X) \arrow{r} & \CH^1_{\Z}(\widetilde D) \arrow{r} & \CH^2_{\Z}(X)
\end{tikzcd}
\end{equation}
   where the arrows are induced by correspondences with integral coefficients. Since numerical and algebraic equivalence agree for codimension-1 cycles on $\widetilde D$, we find that numerical and algebraic equivalence agree on  $\CH^2_{\Z}(X)$. It remains to show that the group of algebraically trivial cycles $\CH_{\Z}^2(X)_{\alg}$ is zero after tensoring with $\Q$. For that purpose, we consider the diagram~\eqref{E:diag-cor} restricted to algebraically trivial cycles. We obtain a commutative diagram
   $$
   \begin{tikzcd}
   	\CH^2_{\Z}(X)_{\alg} \arrow{r} \arrow{d} & \CH^1_{\Z}(\widetilde D)_{\alg} \arrow{r} \arrow{d}{\simeq} & \CH^2_{\Z}(X)_{\alg} \arrow{d} \\
   	\operatorname{Ab}^2_X(\bar K)  \arrow{r} & \operatorname{Pic}^0_{\widetilde D}(\bar K) \arrow{r} & 	\operatorname{Ab}^2_X(\bar K)
   \end{tikzcd}
   $$
   where the composition of the horizontal arrows is given by multiplication by $Nd$, and where the vertical arrows are Murre's algebraic representatives~\cite{Murre-Ab} (these are regular homomorphisms to abelian varieties that are universal). A diagram chase shows that $\CH^2_{\Z}(X)_{\alg} \to 	\operatorname{Ab}^2_X(\bar K) $ is injective after tensoring with $\Q$. We conclude with \cite[Theorem~1.9]{Murre-Ab} which gives the upper bound $\dim 	\operatorname{Ab}^2_X \leq \frac{1}{2}\dim_{\Q_\ell}\HH^3(X_{\bar K},\Q_\ell)$.
\end{proof}

	\subsection{Refined Chow--K\"unneth decomposition}\label{SS:rCK}
	
	Fix a smooth cubic fourfold $X$ over $K$. 
	We are going to produce a refined Chow--K\"unneth decomposition for $X$ that is similar to that for surfaces constructed in \cite[\S 7.2.2]{KMP}, and extending the construction in \cite{BoPe} to arbitrary base fields.  Refining the primitive motive to the transcendental motive is an essential step towards the proof of Theorem~\ref{T:cubics} as it makes it possible to use the ``weight argument'' of Lemma~\ref{lem:weightargument} below. For that purpose,
	recall from Lemma~\ref{L:ratnum} that  $\CH^2(X_{\bar K}) = \overline \CH^2(X_{\bar K})$. This way we can complete $\langle h_X^2 \rangle \subset \CH^2(X)$ to an orthogonal basis $\{ h_{X}^2, \alpha_1, \ldots, \alpha_r\}$ of $\CH^2(X_{\bar K})$ with respect to the intersection product.
 The correspondence 
	\begin{equation}\label{eqn:pi4alg}
	\pi^4_{\alg} := \frac{1}{3} h_X^2\times h_X^2 + \sum_{i=1}^{r} \frac{1}{\deg(\alpha_i\cdot \alpha_i)} \alpha_i \times \alpha_i
	\end{equation} 
	then defines an idempotent in $\CH^4(X_{\bar K}\times_{\bar K} X_{\bar K})$. 
	On the one hand,   the correspondence $\pi^4_{\alg}$ comes from $\CH^2(X_{\bar K}) \otimes \CH^2(X_{\bar K})$ and is Galois-invariant as it defines the intersection pairing on $\CH^2(X_{\bar K})$, and the latter is obviously Galois-invariant. Since we are working with rational coefficients, by  \cite[Example~1.7.6]{MR1644323} and the fact that any cycle is defined over a finite Galois extension of $K$, it follows that $\pi^4_{\alg}$ is defined over $K$, i.e., is in the image of $\CH^4(X\times_{ K} X)$ after base-change to $\bar K$. 
On the other hand, $\pi^4_{\alg}$ commutes with $\pi^4_X$ and $\pi^4_{\prim}$ and is cohomologically the orthogonal projector on the subspace $\operatorname{Im}\big(\CH^2(X_{\bar K}) \to \HH^4(X) \big)$ spanned by $\bar K$-algebraic classes. 
	  In addition, we have $\pi^4_{\alg} \circ \pi^4_{X} = \pi^4_{X} \circ \pi^4_{\alg} = \pi^4_{\alg}$.
	  \medskip
	
	We then define $$\pi^4_{\tr} := \pi^4_X - \pi^4_{\alg}.$$ It is an idempotent correspondence in $\CH^4(X\times_K X)$ which cohomologically is the orthogonal projector on the \emph{transcendental cohomology} $\HH^4_{\tr}(X)$, i.e., by definition of transcendental cohomology, the orthogonal projector on the orthogonal complement to the $\bar K$-algebraic classes in~$\HH^4(X)$. In addition, $\pi^4_{\tr}$ commutes with $\pi^4_{\prim}$ and we have
	\begin{equation}\label{E:commute}
	\pi^4_{\prim} \circ  \pi^4_{\tr} =  \pi^4_{\tr}\circ  \pi^4_{\prim} =  \pi^4_{\tr}.
	\end{equation}
	Note that, while $\pi^4_{\prim}$ is defined family-wise for the universal cubic fourfold $\mathcal{X}\to B$, $\pi^4_{\tr}$ and $\pi^4_{\alg}$ are not.
	\medskip

	Denote by $\h^{i}(X)$, $\h^{4}_{\tr}(X)$ and $\h^{4}_{\alg}(X)$ the Chow motives $(X, \pi^{i}_{X})$, $(X, \pi^{4}_{\tr})$, and $(X, \pi^{4}_{\alg})$ respectively. From the above, we get the following refined Chow--K\"unneth decomposition\,:
	\begin{equation}\label{eqn:WeightDecomp}
	\h(X)=\h^{0}(X)\oplus \h^{2}(X)\oplus\h^{4}_{\alg}(X)\oplus\h^{4}_{\tr}(X)\oplus\h^{6}(X)\oplus\h^{8}(X),
	\end{equation} 
	where $\h^{2i}(X)\simeq \1(-i)$ for $i=0, 1, 3, 4$, the base-change to $\bar K$ of $\h^{4}_{\alg}(X)$ is a direct sum of copies of $\1(-2)$, and $\h^4_{\tr}(X)$ is a direct summand of $\h^4_{\prim}(X)$.\medskip

	As an immediate consequence of \eqref{E:commute}, we have the following consequence of Lemma~\ref{L:key}.
	
	\begin{lem} \label{L:key2}
		The following relations hold in $\CH^4(X\times_K X)$\,:
		$$\pi^4_{\tr} \circ v_4(\P_X^L) \circ \pi^4_{\tr} = \pi^4_{\tr} \quad  \text{and} \quad \pi^4_{\tr} \circ v_4(\P_X^R) \circ \pi^4_{\tr} = \pi^4_{\tr}.$$
		In other words, the correspondences $v_4(\P_X^L)$ and $v_4(\P_X^R)$ act as the identity on the transcendental motive $\h_{\tr}^4(X)$.\qed
	\end{lem}

	\subsection{A weight argument}\label{ss:Weight}
	One defines a notion of \emph{weight} on the  Chow motives appearing in the decomposition~\eqref{eqn:WeightDecomp} in the following way: for any $i\in \Z$, the Tate motive $\1(-i)$ has weight $2i$; $\h^{4}_{\tr}(X)(-i)$ and $\h^{4}_{\alg}(X)(-i)$ have weight $4+2i$.
	As a first step towards our weight argument below (Lemma~\ref{lem:weightargument}), we need the following property of the refined Chow--K\"unneth decomposition~\eqref{eqn:WeightDecomp}.
	\begin{prop}\label{prop:orthogonality}
		Let $X$ and $X'$ be two smooth cubic fourfolds over a field $K$. 
		\begin{enumerate}[$(i)$]
			\item  There is no non-zero morphism from a motive of given weight to a motive of strictly bigger weight among the motives $\1$, $\h^{4}_{\tr}(X)$, $\h^{4}_{\alg}(X)$, $\h^{4}_{\tr}(X')$ and $\h^{4}_{\alg}(X')$ and their Tate twists.
			
			\item $\1(-2)$ and $\h^{4}_{\tr}(X)$ are orthogonal: $\Hom\left(\h^{4}_{\tr}(X), \1(-2)\right)=0$ and $\Hom\left(\1(-2), \h^{4}_{\tr}(X) \right)=0$.
			\item[$(ii')$] $\h^{4}_{\alg}(X')$ and $\h^{4}_{\tr}(X)$ are orthogonal: $$\Hom\left(\h^{4}_{\tr}(X), \h^{4}_{\alg}(X')\right)=0  \quad\text{and}\quad  \Hom\left(\h^{4}_{\alg}(X'), \h^{4}_{\tr}(X) \right)=0.$$

		\end{enumerate}
	\end{prop}
	\begin{proof} Since pull-back by base-change of fields gives injective maps on Chow groups with rational coefficients (by a push-pull argument\,; see, e.g., \cite[Example~1.7.4]{MR1644323}), we may assume $\h^4_{\alg}(X)$ and $\h^4_{\alg}(X')$ are a direct sum of copies of $\1(-2)$. 
		
		The proposition is straightforward to check if one of the motives involved is a Tate motive\,: since $\CH^{l}(X)=\CH^{l}(\h^{2l}(X))$ for $l=0, 1$ and $\CH^{2}(X)=\CH^{2}(\h^{4}_{\alg}(X))$ by construction,  we deduce that for $l\leq 2$, the group $\CH^{l}(\h^{4}_{\tr}(X))$ vanishes, i.e., $\Hom(\1(-l), \h^{4}_{\tr}(X))=0$. Since $\h^{4}_{\tr}(X)^{\vee}= \h^{4}_{\tr}(X)(4)$, we deduce by dualizing that $\Hom(\h^{4}_{\tr}(X), \1(-l))=0$ for $l\geq 2$.
		
			
			It remains to deal with the case where both motives are Tate twists of $\h^{4}_{\tr}(X)$ and $\h^{4}_{\tr}(X')$. Since $\CH_0(\h^{4}_{\tr}(X_\Omega)) = 0$ and $\pi^4_{\tr} = {}^t\pi^4_{\tr}$, we get from
		\cite[Corollary~2.2]{VialCK} that  $\h^{4}_{\tr}(X)(1)$ is isomorphic to a direct summand $N$ of the Chow motive of a surface $S$. Similarly, $\h^{4}_{\tr}(X')(1)$ is isomorphic to a direct summand  of the Chow motive of a surface $S'$.
		As such, we have $\Hom\left(\h^{4}_{\tr}(X), \h^{4}_{\tr}(X')(-l)\right)= \Hom \left(\1(l-2),N\otimes N' \right)$.
		Since $N\otimes N'$ is effective with cohomology concentrated in degree 4, we can then conclude thanks to Lemma~\ref{L:vanishing} below, which is a more general version of \cite[Theorem~1.4(ii)]{FV2} (which states that $\Hom\left(\h^{2}(S), \h^{2}(S')(-l)\right)=0$ for all $l>0$). 
	\end{proof}

	\begin{lem}\label{L:vanishing}
		Let $\HH^*$ be $\ell$-adic cohomology with $\ell \neq \operatorname{char}(K)$. Let $M$ be an effective Chow motive such that $\HH^i(M) =0$ for $i \leq 1$ and such that $\Hom_{\mathcal M_{\hom}} (\1 (-1), M) = 0$ (e.g., $\HH^2(M) = 0$). Then $\CH^l(M) : = \Hom_{\mathcal M} (\1 (-l), M) = 0$ for $l<2$.
	\end{lem}
	\begin{proof}
		By definition of an effective motive, there exists a smooth projective variety $X$ and an idempotent $r\in \operatorname{End}_{\mathcal M}(\h(X))$ such that $M\simeq (X,r,0)$. By assumption, $r$ acts as zero on $\HH^0(X)$, so that $\CH^0(M) := r_*\CH^0(X) = 0$. Further, we have  $\CH^1(M) := r_*\CH^1(X) = 0$ since by assumption $r$ acts as zero both on $\operatorname{Im}(\CH^1(X) \to \HH^2(X))$ and on $\HH^1(X)$ (hence on $\Pic^0_X(K)$).
	\end{proof}
	
	  We will need the following simple observation, which is an abstraction of \cite[\S1.2.3]{FV2}.
	
	\begin{lem}[Weight argument]\label{lem:weightargument}
		Let $\mathcal{S} := \{N_i, i\in I\}$ be a collection of Chow motives whose objects $N_i$ are all equipped with an integer $k_i$ called \emph{weight} such that any morphism from an object of smaller weight to an object of larger weight is zero. For $r = 0,\ldots, n$, let $M_r$ be a Chow motive isomorphic to a direct sum of objects in $\mathcal{S}$.
	Suppose we have a chain of morphisms of Chow motives
		\begin{equation}\label{eqn:chain}
			M=M_{0}\to M_{1} \to M_{2}\to \dots \to M_{n}=M',
		\end{equation}
		such that $M$ and $M'$ are both of (pure) weight $k$ for some integer $k$, i.e., such that $M$ and $M'$ are direct sums of objects of $\mathcal S$ all of weight $k$. Then the composition of morphisms in \eqref{eqn:chain} is equal to the following composition 
		$$ M=M_{0}\to M_{1}^{w=k} \to M_{2}^{w=k}\to \cdots\to M_{n-1}^{w=k}\to M_{n}=M',$$
		where $M_{i}^{w=k}$ means the direct sum of the summands (in $\mathcal S$) of $M_i$ of weight $k$. 
	\end{lem}
	\begin{proof}
		The composition in \eqref{eqn:chain} is clearly the sum of all compositions of the form 
		$$ M=M_{0}\to M_{1}^{w=k_{1}} \to M_{2}^{w=k_{2}}\to \cdots \to M_{n-1}^{w=k_{n-1}}\to M_{n}=M',$$ for $k_{i}\in \Z$.
		However, this composition is non-zero only if $k\geq k_{1}\geq k_{2}\geq \dots \geq k_{n-1}\geq k$ by assumption. Therefore the only non-zero contribution is given by the case where $k_{i}=k$ for all $1\leq i\leq n-1$.
	\end{proof}

	\subsection{Main result}
	\label{Subsec:MainResult}
	Let $X$ and $X'$ be two smooth cubic fourfolds over a field $K$. Assume that their Kuznetsov components $\AA_X$ and $\AA_{X'}$ are \textit{Fourier--Mukai equivalent}, this means there exists an object $\mathcal{E}\in \D^b(X\times_K X')$ such that
$$	\begin{tikzcd}
 F\colon \AA_X \arrow[hook]{r}{i_X} 
 & \D^b(X) \arrow{r}{\Phi_{\mathcal E}} 
 & \D^b(X') \arrow[twoheadrightarrow]{r}{i_{X'}^*} 
 & \AA_{X'} 
	\end{tikzcd}
$$
is an equivalence.  Here $\Phi_{\mathcal E}\colon \D^b(X)\to \D^b(X')$ is the Fourier--Mukai transform associated to the Fourier--Mukai kernel $\mathcal{E}$\,; explicitly,
$$\Phi_{\mathcal{E}}(E):= p_{X', *}(p_X^*(E)\otimes \mathcal{E}),$$
where $p_X$ and $p_{X'}$ are the natural projections from $X\times_K X'$ to $X$ and $X'$ respectively.
Note that by Li--Pertusi--Zhao \cite{LPZ-enhancement22}, over $K=\C$, any equivalence of triangulated categories between $\AA_X$ and $\AA_{X'}$ is a Fourier--Mukai equivalence.

Adding the right adjoints, we get a diagram
$$	\begin{tikzcd}
F: \AA_X \arrow[hook, shift left]{r}{i_X}  \arrow[twoheadleftarrow, shift right, "{i_{X}^!}"']{r}
& \D^b(X) \arrow[shift left]{r}{\Phi_{\mathcal E}} \arrow[leftarrow, shift right, "\Phi_{\mathcal E^R}"']{r} 
& \D^b(X') \arrow[twoheadrightarrow, shift left]{r}{i_{X'}^*}   \arrow[hookleftarrow, shift right, "i_{X'}"']{r}
& \AA_{X'}: F^{R}
\end{tikzcd}
$$
where  $F^R := i_X^! \circ \Phi_{\mathcal E^R}\circ i_{X'}$ denotes the right adjoint functor of $F := i_{X'}^*\circ \Phi_{\mathcal E}\circ i_X$ and where $\mathcal E^R = \mathcal{E}^\vee \otimes^{\mathds{L}} p_X^*\omega_X[4]$ denotes the right adjoint of $\mathcal E$. Since $F$ is an equivalence by assumption, $F^{R}$ is in fact the inverse of $F$, hence we have $F^R\circ F \simeq \mathrm{id}_{\AA_X}$ and $F\circ F^{R} \simeq \mathrm{id}_{\AA_{X'}}$. More explicitly, 
  $$ i_X^! \circ \Phi_{\mathcal E^R}\circ i_{X'} \circ  i_{X'}^*\circ \Phi_{\mathcal E}\circ i_X \simeq \mathrm{id}_{\AA_{X}};$$
  $$ i_{X'}^*\circ \Phi_{\mathcal E}\circ i_X \circ i_X^! \circ \Phi_{\mathcal E^R}\circ i_{X'} \simeq \mathrm{id}_{\AA_{X'}}.$$
These imply that $$i_X\circ i_X^! \circ \Phi_{\mathcal E^R}\circ i_{X'} \circ  i_{X'}^*\circ \Phi_{\mathcal E}\circ i_X\circ i_X^* \simeq i_X\circ i_X^*;$$ 
$$i_{X'}\circ i_{X'}^*\circ \Phi_{\mathcal E}\circ i_X \circ i_X^! \circ \Phi_{\mathcal E^R}\circ i_{X'}\circ i^{!}_{X'} \simeq i_{X'}\circ  i^{!}_{X'}.$$
By definition of the projection functors $p_X^L$ and $p_X^R$ in Section \ref{sec:Projectors}, we have
\begin{eqnarray}\label{E:EER}
\big( p_X^R\circ  \Phi_{\mathcal E^R}\circ p_{X'}^R\big) \circ \big(p_{X'}^L\circ \Phi_{\mathcal E}\circ p_X^L\big) \simeq p_X^L;\\
\label{E:EER2}
\big(p_{X'}^L\circ \Phi_{\mathcal E}\circ p_X^L\big) \circ \big( p_X^R\circ  \Phi_{\mathcal E^R}\circ p_{X'}^R\big) \simeq p_{X'}^R,
\end{eqnarray}
where we have used the isomorphisms $p_{X'}^R \circ p_{X'}^L \simeq p_{X'}^L$ and $p_X^L\circ p_X^R\simeq p_{X}^{R}$ of Proposition~\ref{P:proj}. 
\medskip

Recall that we have defined in \S \S \ref{SS:Kprim}-\ref{SS:rCK} the projectors $\pi^4_{\prim}, \pi^4_{\tr}, \pi^4_{\alg}\in \CH^4(X\times_K X)$ for a cubic fourfold $X$. In the sequel, when dealing with two cubic fourfolds $X$ and $X'$, we keep the same notation for $X$ and use $\pi^4_{\prim'}, \pi^4_{\tr'}, \pi^4_{\alg'}\in \CH^4(X'\times_K X')$  for the corresponding projectors for $X'$. The following is the key step of our proof.
\begin{thm}\label{T:GammaTrans}
The correspondence $\Gamma_{\tr} := \pi^4_{\tr'} \circ v_4(\mathcal E) \circ \pi^4_{\tr}$ in $\CH^4(X\times_K X')$ defines an isomorphism 
$$\begin{tikzcd}
\Gamma_{\tr} : \ \h^4_{\tr}(X) \arrow{r}{\simeq} & \h^4_{\tr}(X')
\end{tikzcd}
$$ with inverse given by its transpose.
In other words, via Proposition~\ref{P:quadratic}, the transcendental motives $\h^4_{\tr}(X)$ and $\h^4_{\tr}(X')$ are isomorphic as quadratic space objects. 
\end{thm}
\begin{proof} 
From the isomorphism of Fourier--Mukai functors  \eqref{E:EER}, it is not clear whether one can deduce an isomorphism between their Fourier--Mukai kernels in $\D^b(X\times X)$, i.e., whether one has an isomorphism
$\big( \P_X^R \ast \mathcal E^R  \ast \P_{X'}^R\big) \ast \big(\P_{X'}^L \ast \mathcal E \ast \P_X^L\big) \simeq \P_X^L$,
where $\ast$ stands for the convolution of Fourier--Mukai kernels. 
Nonetheless, by Canonaco--Stellari \cite[Theorem~1.2]{CS2}, the two sides have the same cohomology sheaves, and hence have the same class in $K_0(X\times X)$. 
By taking Mukai vectors, one obtains
the following equality in $\CH^{*}(X\times_K X)$\,:
\begin{equation}
\label{eqn:EqualAsCycle}
v(\P_X^R)\circ v(\mathcal E^R)\circ v(\P_{X'}^R) \circ v(\P_{X'}^L) \circ v(\mathcal E)\circ v(\P_X^L)= v(\P_X^L).
\end{equation}
The above equality implies that the composition 
\begin{multline*}
\h^{4}_{\tr}(X)\hookrightarrow \h(X)\xrightarrow{v(\P_X^L)} \bigoplus_{i} \h(X)(i)\xrightarrow{ v(\mathcal{E})}\bigoplus_{i}\h(X')(i)\xrightarrow{ v(\P_{X'}^L)} \bigoplus_{i}\h(X')(i)\\
\xrightarrow{ v(\P_{X'}^R)} \bigoplus_{i}\h(X')(i)\xrightarrow{ v(\mathcal{E}^{R})}\bigoplus_{i}\h(X)(i)\xrightarrow{ v(\P_X^R)}\bigoplus_{i}\h(X)(i)\twoheadrightarrow \h^{4}_{\tr}(X)
\end{multline*}
is equal to the composition 
$$\h^{4}_{\tr}(X)\hookrightarrow \h(X)\xrightarrow{ v(\P_X^L)} \bigoplus_{i}\h(X)(i)\twoheadrightarrow \h^{4}_{\tr}(X).$$
Here the ranges of the (finite) direct sums are not specified since they are irrelevant. 
 
 By the ``weight argument'' Lemma \ref{lem:weightargument}, combined with Proposition \ref{prop:orthogonality}$(i)$, we obtain that the composition 
 \begin{multline}\label{eqn:weight4}
\h^{4}_{\tr}(X)\hookrightarrow \h^{4}(X)\xrightarrow{v_{4}(\P_X^L)} \h^{4}(X)\xrightarrow{ v_{4}(\mathcal{E})}\h^{4}(X')\xrightarrow{ v_{4}(\P_{X'}^L)} \h^{4}(X')\\
\xrightarrow{ v_{4}(\P_{X'}^R)} \h^{4}(X')\xrightarrow{ v_{4}(\mathcal{E}^{R})}\h^{4}(X)\xrightarrow{ v_{4}(\P_X^R)}\h^{4}(X)\twoheadrightarrow \h^{4}_{\tr}(X)
\end{multline}
is equal to the composition 
$\h^{4}_{\tr}(X)\hookrightarrow \h^{4}(X)\xrightarrow{ v_{4}(\P_X^L)} \h^{4}(X)\twoheadrightarrow \h^{4}_{\tr}(X)$, which is the identity map of $\h^{4}_{\tr}(X)$ by Lemma \ref{L:key2}. Writing $\h^{4}=\h^{4}_{\tr}\oplus \h^{4}_{\alg}$ and using Proposition \ref{prop:orthogonality}$(ii')$, we deduce that each map in \eqref{eqn:weight4} factors through $\h^{4}_{\tr}$ or $\h^4_{\tr'}$. 
 In other words, we have the following equality:
 \begin{multline*}
\qquad \pi^4_{\tr} \circ v_4(\P_X^R)\circ \pi^4_{\tr} \circ v_4(\mathcal E^R) \circ \pi^4_{\tr'} \circ v_4(\P_{X'}^R)  
\circ \pi^4_{\tr'} \circ v_4(\P_{X'}^L) \circ \pi^4_{\tr'} \circ v_4(\mathcal E) \circ \pi^4_{\tr} \circ v_4(\P_{X}^L) \circ \pi^4_{\tr} = \pi^4_{\tr}.\qquad
 \end{multline*}
By Lemma~\ref{L:key2}, we get 
\begin{equation}\label{eqn:Composition1}
\pi^4_{\tr} \circ v_4(\mathcal E^R) \circ \pi^4_{\tr'} \circ v_4(\mathcal E) \circ \pi^4_{\tr} = \pi^4_{\tr}.
\end{equation} 
Similarly, from \eqref{E:EER2}, together with the weight argument, we obtain 
\begin{equation}\label{eqn:Composition2}
\pi^4_{\tr'} \circ v_4(\mathcal E) \circ \pi^4_{\tr} \circ v_4(\mathcal E^R) \circ \pi^4_{\tr'} = \pi^4_{\tr'}.
\end{equation} 
The equalities \eqref{eqn:Composition1} and \eqref{eqn:Composition2} say nothing but that $\pi^4_{\tr'} \circ v_4(\mathcal E) \circ \pi^4_{\tr}$ and $ \pi^4_{\tr} \circ v_4(\mathcal E^R) \circ \pi^4_{\tr'}$ define inverse isomorphisms between $\h^4_{\tr}(X)$ and $\h^4_{\tr}(X')$.

It remains to show that 
$${}^t\big(\pi^4_{\tr'} \circ v_4(\mathcal E) \circ \pi^4_{\tr} \big)= \pi^4_{\tr} \circ v_4(\mathcal E^R) \circ \pi^4_{\tr'},$$
or equivalently that 
\begin{equation}\label{E:desired}
\pi^4_{\tr} \circ v_4(\mathcal E) \circ \pi^4_{\tr'} = \pi^4_{\tr} \circ v_4(\mathcal E^R) \circ \pi^4_{\tr'}.
\end{equation}
We will actually show the following stronger equality
\begin{equation}\label{E:desired+}
\pi^4_{\prim} \circ v_4(\mathcal E) \circ \pi^4_{\prim'} = \pi^4_{\prim} \circ v_4(\mathcal E^R) \circ \pi^4_{\prim'}.
\end{equation}
To see that \eqref{E:desired+} indeed implies \eqref{E:desired}, it is enough to compose both sides of \eqref{E:desired+} on the left with $\pi^4_{\tr}$ and on the right with $\pi^{4}_{\tr'}$, and then to use \eqref{E:commute}.

Let us show \eqref{E:desired+}. Denoting $h_{X}, h_{X'}\in \CH^{1}(X\times_K X')$ the pull-backs of the hyperplane section classes on  $X$ and $X'$ via the natural projections, we have (see \cite[Lemma 5.41]{MR2244106}) $$v(\mathcal E^R) = v(\mathcal E^\vee \otimes p_X^*\omega_X[4]) =v(\mathcal E^\vee)\cdot \exp(-3h_X)=v(\mathcal E)^{\vee}\cdot \exp\left(\frac{3}{2}(h_{X'}-h_{X})\right).$$ This yields the identity
\begin{multline*}
v_4(\mathcal E^R) = v_4(\mathcal E) + v_3(\mathcal E) \cdot \frac{3}{2}(h_X-h_{X'}) + v_2(\mathcal E)\cdot \frac{(\frac{3}{2})^2}{2!}(h_X-h_{X'})^2 \\+ v_1(\mathcal E)\cdot \frac{(\frac{3}{2})^3}{3!}(h_X-h_{X'})^3 + v_0(\mathcal E)\cdot \frac{(\frac{3}{2})^4}{4!}(h_X-h_{X'})^4.
\end{multline*}
Therefore, to establish \eqref{E:desired+}, it suffices to show the following lemma. 
\begin{lem}\label{lem:hkill}
For any $Z \in \CH^3(X\times_K X')$, we have $\pi^4_{\prim} \circ (Z\cdot h_X) = 0$ and $(Z\cdot h_{X'}) \circ\pi^4_{\prim'}= 0$
\end{lem}
\begin{proof}
We only show the first vanishing; the second one can be proved similarly.
Note that  $\pi^4_{\prim} \circ (Z\cdot h_X) = \pi^4_{\prim} \circ ((\Delta_X)_*(h_{X})) \circ {}^{t}Z$. However, by applying the excess intersection formula  \cite[Theorem~6.3]{MR1644323} 
 to the following cartesian diagram
 with excess normal bundle $\mathcal{O}_{X}(3)$:
 \begin{equation*}
 \xymatrix{
 	X \ar[d]\ar[r]^-{\Delta_{X}}& X\times_K X\ar[d]\\
 	\PP^{5} \ar[r]& \PP^{5}\times_K \PP^{5},
 }
 \end{equation*}
 we obtain that 
$(\Delta_X)_*(3h_{X}) = \Delta_{\PP^{5}}|_{X\times X}= \sum_i h_X^i\times h_X^{5-i}$, where the latter equality uses  the relation $\Delta_{\PP^{5}}=\sum_{i=0}^{5}h^i \times
 h^j$ in
 $\CH^{5}(\PP^{5}\times \PP^{5})$, where $h$ is a
 hyperplane class of $\PP^5$.  We can conclude by noting that for any $i$,  we have $\pi^4_{\prim} \circ (h_X^i\times h_X^{5-i})= 0$ by construction of $\pi^4_{\prim}$.
\end{proof}
With Lemma \ref{lem:hkill} being proved, the equality \eqref{E:desired+}, hence also \eqref{E:desired}, is established. The proof of Theorem~\ref{T:GammaTrans} is complete. 
\end{proof}

\section{Proof of Theorem~\ref{T:cubics}} \label{S:mainproof}

Proposition~\ref{P:cubicsmotives} below, in particular, upgrades the quadratic space object isomorphism of Theorem~\ref{T:GammaTrans} to a quadratic space object isomorphism $\h(X)\simeq \h(X')$.

%
%

\begin{prop}\label{P:cubicsmotives}
	Let $X$ and $X'$ be two smooth cubic fourfolds  over a field $K$, whose Kuznetsov components are Fourier--Mukai equivalent. Then their Chow motives are isomorphic.
	More precisely, there exists a correspondence $\Gamma \in \CH^4(X\times_K X')$ such that $\Gamma_* h_X^i = h_{X'}^i$ for all $i\geq 0$ which in addition induces an isomorphism of Chow motives
	$$\begin{tikzcd}
	\Gamma : \ \h(X) \arrow{r}{\simeq} & \h(X')
	\end{tikzcd}
	$$ with inverse given by its transpose ${}^{t}\Gamma$.
\end{prop}
\begin{proof}
%
As a first step, we construct an isomorphism $\Gamma_{\alg}^{4} : \h^4_{\alg}(X) \to \h^4_{\alg}(X')$ of quadratic space objects. 
Let $\Phi: \mathcal{A}_X \to \mathcal{A}_{X'}$  be the Fourier--Mukai equivalence. It induces a homomorphism 
$$
	\begin{tikzcd}
\overline{\CH}(\mathcal{A}_{X_{\bar{K}}})  \arrow{r}{\simeq}  & \overline{\CH}(\mathcal{A}_{X'_{\bar{K}}}),
\end{tikzcd} \\
\begin{tikzcd}
\bar{v}(E)   \arrow[r, mapsto] & \bar{v}(\Phi(E))
\end{tikzcd}
$$
which is clearly an isometry with respect to the Mukai pairings ($\langle \bar{v}(E), \bar{v}(E')\rangle=\chi(E, E')=\chi(\Phi(E), \Phi(E'))=\langle \bar{v}(\Phi(E)), \bar{v}(\Phi(E'))\rangle$)
and is equivariant with respect to the action of the absolute Galois group of $K$ (since the Fourier--Mukai kernel is defined over $K$). Recall from  Proposition~\ref{P:num} that we have an orthogonal decomposition 
	$$	\overline{\CH}(\mathcal A_{X_{\bar{K}}}) = \langle \lambda_1(h_X),\lambda_2(h_X)\rangle \oPerp \overline{\CH}\,^2_{\prim}(X_{\bar K})$$ 
		with respect to the Mukai pairing.
Since the planes $\langle \lambda_1(h_X),\lambda_2(h_X)\rangle$ and $\langle \lambda_1(h_{X'}),\lambda_2(h_{X'})\rangle$ consist of Galois-invariant elements and are isometric to one another, we obtain from Theorem~\ref{T:Witt}, which is an equivariant Witt theorem,  a Galois-equivariant isometry 
	\[\begin{tikzcd}
\phi \ : \ \overline{\CH}\,^2_{\prim}(X_{\bar K})   \arrow{r}{\simeq}  & \overline{\CH}\,^2_{\prim}(X'_{\bar K}).
\end{tikzcd}
\]
(Note that Theorem \ref{T:Witt} is stated for finite groups, but it indeed applies here: all the numerical Chow groups involved are finitely generated, hence the Galois group action factors through the Galois group of some common finite extension $K'/K$.)
Let then $\{\alpha_1, \ldots, \alpha_r\}$ be an orthogonal basis of $\overline{\CH}\,^2_{\prim}(X_{\bar K}) $.
Having in mind that the Mukai pairing agrees with the intersection pairing on $\overline{\CH}\,^2(X_{\bar K})$ and that $\CH^2(X_{\bar K}) = \overline{\CH}\,^2(X_{\bar K})$, we see, together with the construction and definition of $\h_{\alg}^{4}$ (see \eqref{eqn:pi4alg}),  that the correspondence 
\begin{equation}\label{E:gamma4alg}
\Gamma_{\alg}^{4}:=\frac{1}{3}h_{X}^{2}\times h_{X'}^{2}+\sum_{i=1}^{r}\frac{1}{\deg(\alpha^{2}_{i})} \alpha_{i}\times \phi(\alpha_{i})\quad \in \CH^{4}(X_{\bar{K}}\times_{\bar K} X'_{\bar K})
\end{equation}
%
%
is defined over $K$
 and defines an isomorphism $\h^4_{\alg}(X) \stackrel{\simeq}{\longrightarrow} \h^4_{\alg}(X')$ with inverse given by its transpose~${}^{t}\Gamma_{\alg}^{4}$.

	Finally, combining $\Gamma_{\alg}^{4}$ with $\Gamma_{\tr}$ of Theorem~\ref{T:GammaTrans}, the cycle $$\Gamma:=\frac{1}{3}h_{X}^{4}\times X'+ \frac{1}{3}h_{X}^{3}\times h_{X'}+\Gamma_{\alg}^{4}+\Gamma_{\tr}+ \frac{1}{3}h_{X}\times h_{X'}^{3}+\frac{1}{3}X\times h_{X'}^{4}\in \CH^{4}(X\times X')$$ induces an isomorphism between $\h(X)$ and $\h(X')$, and its inverse is ${}^{t}\Gamma$. Furthermore, by construction, we have $\Gamma_{*}(h_{X}^{i})=h_{X'}^{i}$ for all $i$.
\end{proof}

\begin{rmk}\label{R:complex}
In the case where $K=\C$ and $\HH^*$ is Betti cohomology, the construction of the isomorphism $\Gamma^4_{\alg} : \h^4_{\alg}(X) \to \h^4_{\alg}(X')$ in the proof of Proposition~\ref{P:cubicsmotives} is somewhat simpler. 
As a consequence of Theorem~\ref{T:GammaTrans}, we have a Hodge isometry
\begin{equation}\label{E:trC}
\HH^{4}_{\tr}(X, \Q) \simeq\HH^{4}_{\tr}(X', \Q).
\end{equation}
(This Hodge isometry can also be obtained by considering the transcendental part of \cite[Proposition~3.4]{MR3705236}.)
Since $\HH^4(X,\Q)$ and $\HH^4(X',\Q)$ are isometric for all smooth complex cubic fourfolds, there is by Witt's theorem an isometry 
\begin{equation}\label{E:algC}
\phi: \HH^{4}_{\alg}(X, \Q)\xrightarrow{\simeq} \HH^{4}_{\alg}(X', \Q)
\end{equation}
sending $h_{X}^{2}$ to $h_{X'}^{2}$. Let $\{h_{X}^{2}, \alpha_{1}, \dots, \alpha_{r}\}$ be an orthogonal basis of $\HH^{4}_{\alg}(X, \Q)$.
 The correspondence $\Gamma^4_{\alg}$ of~\eqref{E:gamma4alg} then provides an isomorphism from 
$\h_{\alg}^{4}(X)$ to $ \h_{\alg}^{4}(X')$, whose inverse is given by its transpose ${}^{t}\Gamma_{\alg}^{4}$. Note that, by combining \eqref{E:trC} and \eqref{E:algC}, we obtain a Hodge isometry $\HH^4(X,\Q) \simeq \HH^4(X',\Q)$.
\end{rmk}

Theorem~\ref{T:cubics} then follows from combining Proposition~\ref{P:cubicsmotives} with the following proposition.
\begin{prop} \label{P:upgradeFrob}
	Let $X$ and $X'$ be two smooth cubic fourfolds. Assume that there exists a correspondence $\Gamma \in \CH^4(X\times_K X')$ such that $\Gamma_* h_X^i = h_{X'}^i$ for all $i\geq 0$ which in addition induces an isomorphism 
$$\begin{tikzcd}
\Gamma : \ \h(X) \arrow{r}{\simeq} & \h(X')
\end{tikzcd}
$$ with inverse given by its transpose. Then $\Gamma$ is an isomorphism of Chow motives, as Frobenius algebra objects.
\end{prop}
\begin{proof}
Recall in general \cite[Proposition~2.11]{FV2} that a morphism $\Gamma: \h(X) \to \h(X')$ between the Chow motives of smooth projective varieties of same dimension is an isomorphism of Chow motives, as Frobenius algebra objects, if $\Gamma$ is an isomorphism of Chow motives, $(\Gamma\otimes \Gamma)_*\Delta_X = \Delta_{X'}$ and $(\Gamma \otimes \Gamma \otimes \Gamma)_* \delta_X =\delta_{X'}$, where $\delta$ denotes the small diagonal.
Let now $\Gamma$ be as in the statement of the proposition. That $\Gamma$ defines an isomorphism with inverse given by its transpose is equivalent to $\Gamma$ is an isomorphism and $(\Gamma\otimes \Gamma)_*\Delta_X = \Delta_{X'}$. Therefore, we only need to check that 
 $$(\Gamma \otimes \Gamma \otimes \Gamma)_* \delta_X =\delta_{X'}.$$
 However, by Theorem~\ref{thm:MCKcubic}, and using the assumption that $\Gamma_* h_X^i = h_{X'}^i$ for all $i\geq 0$, we have 
 	\begin{align*}
(\Gamma \otimes \Gamma \otimes \Gamma)_*   \delta_X = & \  \frac{1}{3} \big( p_{12}^*(\Gamma \otimes \Gamma)_* \Delta_X\cdot p_3^*h_{X'}^4 +   \operatorname{perm.} \big)+  P\big(p_{1}^*h_{X'}, p_2^*h_{X'}, p_3^*h_{X'} \big)\\
 = & \  \frac{1}{3} \big( p_{12}^*\Delta_{X'}\cdot p_3^*h_{X'}^4 +  \operatorname{perm.} \big)+  P\big(p_{1}^*h_{X'}, p_2^*h_{X'}, p_3^*h_{X'} \big)\\
 = & \ \delta_{X'},
 \end{align*} 
 where in the second equality we have used the identity 
 $(\Gamma\otimes \Gamma)_*\Delta_X = \Delta_{X'}$.
\end{proof}

\section{Cubic fourfolds with associated K3 surfaces}

Let $X$ be a smooth cubic fourfold over a field $K$ and let $\AA_{X}$ be the Kuznetsov component of $\D^{b}(X)$ as before. Assume that there exists a K3 surface $S$ endowed with a Brauer class $\alpha\in \operatorname{Br}(X)$, such that $\AA_{X}$ is Fourier--Mukai equivalent to $\D^{b}(S,\alpha)$. 
That is, there exists an object $\mathcal{E}\in D^{b}(X\times S, 1\times \alpha)$, such that the composition
\begin{equation*}
\begin{tikzcd}
\AA(X) \arrow[hook]{r}{i_{X}} &\D^{b}(X) \arrow{r}{\Phi_{\mathcal{E}}} &\D^{b}(S, \alpha)
\end{tikzcd}
\end{equation*}
is an equivalence of triangulated categories, where $i_{X}$ is the natural inclusion. The goal of this section is to prove Theorem \ref{T:cubicK3}. The proof is similar to that of Theorem~\ref{T:cubics-quadratic} and we will only sketch the main steps. 
 In the sequel, let us omit $\alpha$ from the notation, since the proof for the twisted case is the same as the untwisted case. \medskip

The right adjoint of the functor $\Phi_{\mathcal{E}}\circ i_{X}$ is $i_{X}^{!}\circ\Phi_{\mathcal{E}^{R}}$. Hence the hypothesis implies that 
\begin{equation*}
 i_{X}^{!}\circ\Phi_{\mathcal{E}^{R}}\circ \Phi_{\mathcal{E}}\circ i_{X}\simeq \operatorname{id}_{\AA_{X}}\,;
 \end{equation*}
 \begin{equation*}
  \Phi_{\mathcal{E}}\circ i_{X}\circ i_{X}^{!}\circ\Phi_{\mathcal{E}^{R}}\simeq \operatorname{id}_{\D^{b}(S)}.
 \end{equation*}
By the definition of $p_X^L$ and $p_X^R$ in Section \ref{sec:Projectors}, we obtain
\begin{equation}\label{eqn:FunctorToCycle1}
p^{R}_{X}\circ\Phi_{\mathcal{E}^{R}}\circ \Phi_{\mathcal{E}}\circ p_X^L\simeq p_X^L;
 \end{equation}
 \begin{equation}\label{eqn:FunctorToCycle2}
  \Phi_{\mathcal{E}}\circ p_X^L\circ p_X^R\circ\Phi_{\mathcal{E}^{R}}\simeq \operatorname{id}_{\D^{b}(S)}.
 \end{equation}
 Recall that $\mathcal{P}^{L}_{X}, \mathcal{P}^{R}_{X}\in \D^{b}(X\times_K X)$ are the Fourier--Mukai kernels of the functors $p_X^L$ and $p_X^R$ respectively. 
 As in the proof of Theorem \ref{T:GammaTrans}, using \cite[Theorem~6.4]{CS2}, we deduce from \eqref{eqn:FunctorToCycle1} that in $\CH^*(X\times_K X)$,
 \begin{equation}\label{eqn:ComposeToIdMixedDeg1}
 v(\mathcal{P}^{R}_{X})\circ v(\mathcal{E}^{R})\circ v(\mathcal{E})\circ v(\mathcal{P}^{L}_{X})=v(\mathcal{P}^{L}_{X}),
 \end{equation}
where $v$ denotes the Chow-theoretic Mukai vector map. 
Likewise, using \cite[Theorem~6.4]{CS2}, or alternately by the uniqueness of the Fourier--Mukai kernel in the twisted version of Orlov's Theorem (\cite[Theorem~1.1]{CS1}), \eqref{eqn:FunctorToCycle2} implies that 
\begin{equation}\label{eqn:ComposeToIdMixedDeg2}
v(\mathcal{E})\circ v(\mathcal{P}^{L}_{X})\circ v(\mathcal{P}^{R}_{X})\circ v(\mathcal{E}^{R})=\Delta_{S}.
\end{equation}
\smallskip

As in Section~\ref{S:FM-tr}, we define a refined Chow--K\"unneth decomposition for $S$. The general case of a smooth projective surface over $K$ is due to \cite[\S 7.2.2]{KMP}. Since for a K3 surface rational and numerical equivalence agree on $\CH^1(S_{\bar K})$, we can in fact construct such a refined Chow--K\"unneth decomposition in a more direct way. First, choose any degree-1 zero-cycle $o\in \CH_0(S)$, and define the Chow--K\"unneth decomposition
$$\pi^0_S := o\times S, \quad \pi^4_S := S\times o, \quad \mbox{and} \quad \pi^2_S := \Delta_S - \pi^0_S - \pi^4_S.$$
 Let $\{\beta_1,\ldots, \beta_s\}$ be an orthogonal basis for $\CH^1(S_{\bar K})$. 
 The correspondence 
\begin{equation}\label{eqn:pi2alg}
\pi^2_{\alg,S} := \sum_{i=1}^{s} \frac{1}{\deg(\beta_i\cdot \beta_i)} \beta_i \times \beta_i
\end{equation} then defines an idempotent in $\CH^2(S_{\bar K}\times_{\bar K} S_{\bar K})$ which descends
to $K$, which commutes with $\pi^2_S$  and which cohomologically is the orthogonal projector on the subspace $\operatorname{Im}\big(\CH^1(S_{\bar K}) \to \HH^2(S)\big)$ spanned by $\bar K$-algebraic classes in $\HH^2(S)$. In addition, we have $\pi^2_{\alg,S} \circ \pi^2_{S} = \pi^2_{S} \circ \pi^2_{\alg,S} = \pi^2_{\alg,S}$.

We then define $$\pi^2_{\tr,S} := \pi^2_S - \pi^2_{\alg,S}.$$ It is an idempotent correspondence in $\CH^2(S\times_K S)$ which cohomologically is the orthogonal projector on the \emph{transcendental cohomology} $\HH^2_{\tr}(S)$, i.e., by definition of transcendental cohomology, the orthogonal projector on the orthogonal complement to the $\bar K$-algebraic classes in $\HH^2(S)$.

Denote by $\h^{i}(S)$, $\h^{2}_{\tr}(S)$ and $\h^{2}_{\alg}(S)$ the Chow motives $(S, \pi^{i}_{S})$, $(S, \pi^{2}_{\tr,S})$, and $(S, \pi^{2}_{\alg, S})$ respectively. From the above, we get the following refined Chow--K\"unneth decomposition\,:
\begin{equation*}\label{eqn:WeightDecompS}
\h(S)=\h^{0}(S)\oplus \h^{2}_{\alg}(S)\oplus\h^{2}_{\tr}(S)\oplus\h^{4}(S),
\end{equation*} 
where $\h^{2i}(X)\simeq \1(-i)$ for $i=0, 2$ and the base-change to $\bar K$ of $\h^{2}_{\alg}(S)$ is a direct sum of copies of $\1(-1)$.
\medskip

Now, as in the case of two cubic fourfolds, we want to apply the weight argument (Lemma~\ref{lem:weightargument}) to the equalities \eqref{eqn:ComposeToIdMixedDeg1} and \eqref{eqn:ComposeToIdMixedDeg2}. To this end, we need the following complement to Proposition~\ref{prop:orthogonality}.

\begin{prop}\label{prop:orthogonality2}
Let $X$ be a cubic fourfold and $S$ a projective surface. 
Then for all $l>1$,
$$\Hom\left(\h^{4}_{\tr}(X), \h^{2}_{\tr}(S)(-l)\right)=0.$$
\end{prop}
\begin{proof}
As is pointed out in the proof of Proposition \ref{prop:orthogonality}, $\h^{4}_{\tr}(X)(1)$ is a direct summand of the motive of a surface. Then we can apply Lemma \ref{L:vanishing} to conclude to the vanishing. 
\end{proof}

By the weight argument (Lemma \ref{lem:weightargument}), combined with Proposition \ref{prop:orthogonality},  \cite[Theorem~1.4(ii)]{FV2} and Proposition \ref{prop:orthogonality2}, we can deduce that if we restrict the domain to $\h_{\tr}^{4}(X)$, then each step of \eqref{eqn:ComposeToIdMixedDeg1} factors through $\h_{\tr}^{4}(X)$ or $\h^{2}_{\tr}(S)(-1)$. In other words,
\begin{equation*}\label{eqn:ComposeToIdPureDeg1}
 \pi^4_{\tr, X} \circ v_4(\P_X^R)\circ \pi^4_{\tr, X}  \circ v_3(\mathcal E^R) \circ \pi^2_{\tr, S} \circ v_3(\mathcal E) \circ \pi^4_{\tr, X} \circ v_4(\P_{X}^L) \circ \pi^4_{\tr, X} = \pi^4_{\tr, X}.
\end{equation*} 
By Lemma~\ref{L:key2}, we get
\begin{equation}\label{eqn:ComposeToIdPureDeg1'}
 \pi^4_{\tr, X} \circ v_3(\mathcal E^R) \circ \pi^2_{\tr, S} \circ v_3(\mathcal E) \circ \pi^4_{\tr, X}= \pi^4_{\tr,X}.
\end{equation}
Similarly, \eqref{eqn:ComposeToIdMixedDeg2} implies 
\begin{equation}\label{eqn:ComposeToIdPureDeg2}
 \pi^2_{\tr, S} \circ v_3(\mathcal E) \circ \pi^4_{\tr, X} \circ v_3(\mathcal E^R) \circ \pi^2_{\tr, S}=\pi^2_{\tr, S}.
\end{equation}
Note that \eqref{eqn:ComposeToIdPureDeg1'} and \eqref{eqn:ComposeToIdPureDeg2} together say that we have the following pair of inverse isomorphisms:
\begin{equation}\label{eqn:InverseIso}
 \begin{tikzcd}
 \h^{4}_{\tr}(X) \arrow[shift left]{rrr}{\pi^2_{\tr, S} \circ v_3(\mathcal E) \circ \pi^4_{\tr, X}} &  && \h^{2}_{\tr}(S)(-1)\arrow[shift left]{lll}{\pi^4_{\tr, X} \circ v_3(\mathcal E^R) \circ \pi^2_{\tr, S} }
 \end{tikzcd}
\end{equation}
By the same argument as in the proof of \eqref{E:desired}, using Lemma \ref{lem:hkill}, we can moreover show that the two inverse isomorphisms in \eqref{eqn:InverseIso} are transpose to each other. To summarize, we have proven the following\,:

\begin{thm}\label{T:cubicK3v3}
The correspondence $\Gamma_{\tr} := \pi^2_{\tr,S} \circ v_3(\mathcal E) \circ \pi^4_{\tr,X}$ in $\CH^3(X\times S)$ induces an isomorphism 
$$\begin{tikzcd}
\Gamma_{\tr} : \ \h^4_{\tr}(X)(2) \arrow{r}{\simeq} & \h^2_{\tr}(S)(1)
\end{tikzcd}
$$ whose inverse is its transpose ${}^{t}\Gamma_{\tr}$.\qed
\end{thm}

Via Proposition~\ref{P:quadratic}, Theorem~\ref{T:cubicK3v3} establishes Theorem~\ref{T:cubicK3}.\qed


\appendix
\section{An equivariant Witt theorem}
Throughout the appendix,  $F$ is a field of characteristic different from 2 and all the vector spaces are finite dimensional over $F$. 

Let us first recall the classical Witt theorem. 
Let $V_1, V_2$ be vector spaces equipped with quadratic forms, whose associated bilinear symmetric pairings are denoted by $\langle-, -\rangle$. Suppose that $V_1$ and $V_2$ are isometric and we have orthogonal decompositions 
\[V_1=U_1\oPerp W_1, \quad V_2=U_2\oPerp W_2,\]
such that $U_1$ and $U_2$ are isometric. Then $W_1$ and $W_2$ are also isometric.
This is often referred to as \textit{Witt's cancellation theorem}, which is clearly equivalent to the following \textit{Witt's extension theorem}\,: Let $V$ be a non-degenerate quadratic space and let $f: U\to U'$ be an isometry between two subspaces of $V$. Then $f$ can be extended to an isometry of $V$.
\medskip

The goal of this appendix is to establish an equivariant version of the Witt theorem, in case the quadratic spaces are endowed with a group action. For a quadratic space $V$ with a $G$-action, we denote $O_G(V)$ the group of $G$-equivariant isometries, i.e.~ automorphisms of $V$ that preserve the pairing and commute with the action of $G$.

\begin{lem}\label{lem:transitive}
	Let $V$ be a non-degenerate quadratic space equipped with an isometric action of a finite group $G$. Suppose that $|G|$ is invertible in $F$. Then
	\begin{enumerate}
		\item The restriction of the quadratic form to $V^G$, the $G$-fixed space, is non-degenerate.
		\item For any $x, y\in V^G$ with $\langle x,x\rangle =\langle y,y\rangle\neq 0$, there exists a $G$-equivariant isometry $\phi\in O_G(V)$ sending $x$ to $y$.
	\end{enumerate}
\end{lem}
\begin{proof}
	For (1), let $x\in \operatorname{rad}(V^G)$,
	 for any $y\in V$, 
	\[
	\langle x, y\rangle=\frac{1}{|G|}\sum_{g\in G}\langle gx,gy\rangle=\langle x,\frac{1}{|G|}\sum_{g\in G}gy\rangle=0,
	\]
	since $\frac{1}{|G|}\sum_{g\in G}gy\in V^G$. Therefore, $x\in \operatorname{rad}(V)=\{0\}$.\\
	For (2), as $x$ and $y$ are anisotropic,  it is well-known that there exists $\phi_1\in O(V^G)$, a reflection or a product of two reflections, which sends $x$ to $y$. By (1), we have an orthogonal decomposition 
	\[V=V^G\oplus (V^G)^\perp.\] Hence we can take $\phi:=\phi_1\oplus \operatorname{id}_{(V^G)^\perp}$.
\end{proof}

\begin{thm}\label{T:Witt}
	Let $V_1$, $V_2$ be two non-degenerate quadratic spaces endowed with actions of a finite group $G$ by isometries. Assume that $|G|$ is invertible in the base field $F$. Suppose that we have orthogonal decompositions preserved by $G$:
	\[V_1=U_1\oPerp W_1, \quad V_2=U_2\oPerp W_2,\] 
	satisfying the following conditions:
	\begin{itemize}
		\item there is a $G$-equivariant isometry between $V_1$ and $V_2$;
		\item $W_1\subset V_1^G$ and $W_2\subset V_2^G$;
		\item $W_1$ and $W_2$ are isometric.
	\end{itemize}
	Then there exists a $G$-equivariant isometry between $U_1$ and $U_2$.
\end{thm}
\begin{proof} We only give a proof in the case where $W_1$ and $W_2$ are assumed to be non-degenerate\,; the general case (which we do not use in this paper) is left to the reader.
	We may and will identify $W_1$ and $W_2$, and denote both $W$. Let us first treat the case where $W$ is of dimension 1, generated by a vector $x$ with $\langle x,x\rangle\neq 0$. By hypothesis, there is a $G$-equivariant isometry 
	\[ V_1=Fx\oplus U_1\stackrel{\phi}{\longrightarrow} V_2=Fx\oplus U_2.\]
	Denote $y=\phi(x)$ and $U_1'=\phi(U_1)$. Hence $0\neq \langle x,x\rangle =\langle y,y\rangle$ and $x, y$ are both $G$-invariant. Applying Lemma \ref{lem:transitive}, we get a $G$-equivariant isometry $\tau\in O_G(V_2)$ sending $x$ to $y$. Therefore $\tau(U_2)$,  being orthogonal to $y$, must be $U_1'$. In particular, $U_2$ is $G$-equivariantly isometric to $U_1'$, hence also to $U_1$.
	
	In the general case, we diagonalize $W$ and proceed by induction.
\end{proof}

\bibliographystyle{amsalpha}
\bibliography{bib}

\newcommand{\etalchar}[1]{$^{#1}$}
\providecommand{\bysame}{\leavevmode\hbox to3em{\hrulefill}\thinspace}
\providecommand{\MR}{\relax\ifhmode\unskip\space\fi MR }
\providecommand{\MRhref}[2]{%
  \href{http://www.ams.org/mathscinet-getitem?mr=#1}{#2}
}
\providecommand{\href}[2]{#2}
\begin{thebibliography}{BLM{\etalchar{+}}17}

\bibitem[And04]{MR2115000}
Yves Andr{\'e}, \emph{Une introduction aux motifs (motifs purs, motifs mixtes,
  p{\'e}riodes)}, Panoramas et Synth{\`e}ses [Panoramas and Syntheses],
  vol.~17, Soci{\'e}t{\'e} Math{\'e}matique de France, Paris, 2004. \MR{2115000
  (2005k:14041)}

\bibitem[AT14]{MR3229044}
Nicolas Addington and Richard Thomas, \emph{Hodge theory and derived categories
  of cubic fourfolds}, Duke Math. J. \textbf{163} (2014), no.~10, 1885--1927.
  \MR{3229044}

\bibitem[AV20]{AntieauVezzosi-HKR}
Benjamin Antieau and Gabriele Vezzosi, \emph{A remark on the
  {H}ochschild-{K}ostant-{R}osenberg theorem in characteristic {$p$}}, Ann. Sc.
  Norm. Super. Pisa Cl. Sci. (5) \textbf{20} (2020), no.~3, 1135--1145.
  \MR{4166802}

\bibitem[BK89]{BondalKapranov}
A.~I. Bondal and M.~M. Kapranov, \emph{Representable functors, {S}erre
  functors, and reconstructions}, Izv. Akad. Nauk SSSR Ser. Mat. \textbf{53}
  (1989), no.~6, 1183--1205, 1337. \MR{1039961}

\bibitem[BLM{\etalchar{+}}17]{BLMS}
Arend Bayer, Mart\'{i} Lahoz, Emanuele Macr\`{i}, Paolo Stellari, and Xiaolei
  Zhao, \emph{Stability conditions on {K}uznetsov components}, arXiv:
  1703.10839 (2017).

\bibitem[BO01]{BondalOrlov01}
Alexei Bondal and Dmitri Orlov, \emph{Reconstruction of a variety from the
  derived category and groups of autoequivalences}, Compositio Math.
  \textbf{125} (2001), no.~3, 327--344. \MR{1818984}

\bibitem[Bon89]{MR992977}
A.~I. Bondal, \emph{Representations of associative algebras and coherent
  sheaves}, Izv. Akad. Nauk SSSR Ser. Mat. \textbf{53} (1989), no.~1, 25--44.
  \MR{992977}

\bibitem[BP20]{BoPe}
Michele Bolognesi and Claudio Pedrini, \emph{The transcendental motive of a
  cubic fourfold}, J. Pure Appl. Algebra \textbf{224} (2020), no.~8, 106333,
  16. \MR{4074573}

\bibitem[BS83]{BS}
S.~Bloch and V.~Srinivas, \emph{Remarks on correspondences and algebraic
  cycles}, Amer. J. Math. \textbf{105} (1983), no.~5, 1235--1253. \MR{714776}

\bibitem[B{\"u}l20]{Buelles}
Tim-Henrik B{\"u}lles, \emph{Motives of moduli spaces on {K}3 surfaces and of
  special cubic fourfolds}, Manuscripta Math. \textbf{161} (2020), no.~1-2,
  109--124. \MR{4046978}

\bibitem[CS07]{CS1}
Alberto Canonaco and Paolo Stellari, \emph{Twisted {F}ourier-{M}ukai functors},
  Adv. Math. \textbf{212} (2007), no.~2, 484--503. \MR{2329310}

\bibitem[CS12]{CS2}
\bysame, \emph{Non-uniqueness of {F}ourier-{M}ukai kernels}, Math. Z.
  \textbf{272} (2012), no.~1-2, 577--588. \MR{2968243}

\bibitem[Dia21]{Diaz}
H.~Anthony Diaz, \emph{The {C}how ring of a cubic hypersurface}, Int. Math.
  Res. Not. IMRN (2021), no.~22, 17071--17090. \MR{4345821}

\bibitem[DV10]{MR2746467}
Olivier Debarre and Claire Voisin, \emph{Hyper-{K}{\"a}hler fourfolds and
  {G}rassmann geometry}, J. Reine Angew. Math. \textbf{649} (2010), 63--87.
  \MR{2746467}

\bibitem[FLV21a]{FLV3}
Lie Fu, Robert Laterveer, and Charles Vial, \emph{The generalized {F}ranchetta
  conjecture for some hyper-{K}\"{a}hler varieties, {II}}, J. \'{E}c. polytech.
  Math. \textbf{8} (2021), 1065--1097. \MR{4263794}

\bibitem[FLV21b]{FLV2}
\bysame, \emph{Multiplicative {C}how-{K}\"{u}nneth decompositions and varieties
  of cohomological {K}3 type}, Ann. Mat. Pura Appl. (4) \textbf{200} (2021),
  no.~5, 2085--2126. \MR{4285110}

\bibitem[Ful98]{MR1644323}
William Fulton, \emph{Intersection theory}, second ed., Ergebnisse der
  Mathematik und ihrer Grenzgebiete. 3. Folge. A Series of Modern Surveys in
  Mathematics [Results in Mathematics and Related Areas. 3rd Series. A Series
  of Modern Surveys in Mathematics], vol.~2, Springer-Verlag, Berlin, 1998.
  \MR{1644323 (99d:14003)}

\bibitem[FV21]{FV2}
Lie {Fu} and Charles {Vial}, \emph{{A motivic global Torelli theorem for
  isogenous K3 surfaces}}, Adv. Math. \textbf{383} (2021).

\bibitem[Huy06]{MR2244106}
Daniel Huybrechts, \emph{Fourier-{M}ukai transforms in algebraic geometry},
  Oxford Mathematical Monographs, The Clarendon Press, Oxford University Press,
  Oxford, 2006. \MR{2244106}

\bibitem[Huy17]{MR3705236}
\bysame, \emph{The {K}3 category of a cubic fourfold}, Compos. Math.
  \textbf{153} (2017), no.~3, 586--620. \MR{3705236}

\bibitem[Huy19]{Huy2019Lecture}
\bysame, \emph{{Hodge theory of cubic fourfolds, their Fano varieties, and
  associated K3 categories}}, Lecture Notes of the Unione Matematica Italiana
  \textbf{26} (2019), 165--198.

\bibitem[KMP07]{KMP}
Bruno Kahn, Jacob~P. Murre, and Claudio Pedrini, \emph{On the transcendental
  part of the motive of a surface}, Algebraic cycles and motives. {V}ol. 2,
  London Math. Soc. Lecture Note Ser., vol. 344, Cambridge Univ. Press,
  Cambridge, 2007, pp.~143--202. \MR{2187153}

\bibitem[Koc04]{MR2037238}
Joachim Kock, \emph{Frobenius algebras and 2{D} topological quantum field
  theories}, London Mathematical Society Student Texts, vol.~59, Cambridge
  University Press, Cambridge, 2004. \MR{2037238}

\bibitem[KP18]{KP16}
Alexander Kuznetsov and Alexander Perry, \emph{Derived categories of
  {G}ushel-{M}ukai varieties}, Compositio Mathematica \textbf{154} (2018),
  1362--1406.

\bibitem[Kuz10]{Kuz10}
Alexander Kuznetsov, \emph{Derived categories of cubic fourfolds}, Progress in
  Mathematics, vol. 282, pp.~219--243, Birkh\"{a}user Boston, 2010.

\bibitem[Kuz16]{Kuznetsov-LectureNotes}
\bysame, \emph{Derived categories view on rationality problems}, Rationality
  problems in algebraic geometry, Lecture Notes in Math., vol. 2172, Springer,
  Cham, 2016, pp.~67--104. \MR{3618666}

\bibitem[Kuz19]{Kuznetsov-FractionalCY}
\bysame, \emph{Calabi-{Y}au and fractional {C}alabi-{Y}au categories}, J. Reine
  Angew. Math. \textbf{753} (2019), 239--267. \MR{3987870}

\bibitem[LLMS18]{LLMS}
Mart\'{\i} Lahoz, Manfred Lehn, Emanuele Macr\`\i, and Paolo Stellari,
  \emph{Generalized twisted cubics on a cubic fourfold as a moduli space of
  stable objects}, J. Math. Pures Appl. (9) \textbf{114} (2018), 85--117.
  \MR{3801751}

\bibitem[LPZ22]{LPZ19}
Chunyi Li, Laura Pertusi, and Xiaolei Zhao, \emph{Elliptic quintics on cubic
  fourfolds, {O}'{G}rady 10, and {L}agrangian fibrations}, Adv. Math.
  \textbf{408} (2022), no.~part A, Paper No. 108584, 56. \MR{4456791}

\bibitem[LPZ23a]{LPZ-enhancement22}
\bysame, \emph{Derived categories of hearts on {K}uznetsov components}, to
  appear in Journal of the London Mathematical Society, arXiv:2203.13864
  (2023).

\bibitem[LPZ23b]{LPZ18}
\bysame, \emph{Twisted cubics on cubic fourfolds and stability conditions}, to
  appear in Algebraic Geometry, arXiv:1802.01134 (2023).

\bibitem[Muk89]{Muk89}
Shigeru Mukai, \emph{Biregular classification of {F}ano {$3$}-folds and {F}ano
  manifolds of coindex {$3$}}, Proc. Nat. Acad. Sci. U.S.A. \textbf{86} (1989),
  no.~9, 3000--3002.

\bibitem[Mur85]{Murre-Ab}
J.~P. Murre, \emph{Applications of algebraic {$K$}-theory to the theory of
  algebraic cycles}, Algebraic geometry, {S}itges ({B}arcelona), 1983, Lecture
  Notes in Math., vol. 1124, Springer, Berlin, 1985, pp.~216--261. \MR{805336}

\bibitem[Mur93]{Murre}
\bysame, \emph{On a conjectural filtration on the {C}how groups of an algebraic
  variety. {I}. {T}he general conjectures and some examples}, Indag. Math.
  (N.S.) \textbf{4} (1993), no.~2, 177--188. \MR{1225267}

\bibitem[Orl03]{MR1998775}
D.~O. Orlov, \emph{Derived categories of coherent sheaves and equivalences
  between them}, Uspekhi Mat. Nauk \textbf{58} (2003), no.~3(351), 89--172.
  \MR{1998775}

\bibitem[SV16]{SV}
Mingmin Shen and Charles Vial, \emph{The {F}ourier transform for certain
  hyperk\"ahler fourfolds}, Mem. Amer. Math. Soc. \textbf{240} (2016),
  no.~1139, vii+163. \MR{3460114}

\bibitem[Via15]{VialCK}
Charles Vial, \emph{Chow-{K}\"{u}nneth decomposition for 3- and 4-folds fibred
  by varieties with trivial {C}how group of zero-cycles}, J. Algebraic Geom.
  \textbf{24} (2015), no.~1, 51--80. \MR{3275654}

\bibitem[Yek02]{Yikutieli}
Amnon Yekutieli, \emph{The continuous {H}ochschild cochain complex of a
  scheme}, Canad. J. Math. \textbf{54} (2002), no.~6, 1319--1337. \MR{1940241}

\end{thebibliography}
\end{document}